\documentclass[a4paper,reqno,twoside,10pt]{amsart}

\usepackage{amsmath,amssymb,amsthm,mathtools}
\usepackage{booktabs,array,longtable,threeparttable}
\usepackage{capt-of}
\usepackage{enumitem}
\usepackage[hidelinks]{hyperref}
\usepackage[T1]{fontenc}
\usepackage{microtype}
\usepackage{newtxtext,newtxmath}
\usepackage{helvet}
\usepackage{courier}

\usepackage{tikz}
\usetikzlibrary{calc}

\tikzset{
  graph/.style={x=0.9cm,y=0.9cm,line width=0.5pt,baseline=(current bounding box.center)},
}

\newcommand{\vtx}[1]{\fill (#1) circle (2pt);}
\newcommand{\chainarm}[2]{%
  \draw (#1) -- ($(#1)!0.36!(#2)$);
  \draw ($(#1)!0.64!(#2)$) -- (#2);
  \foreach \t in {0.44,0.5,0.56}{\fill ($(#1)!\t!(#2)$) circle (0.55pt);}%
}

\setlist[enumerate,1]{label=(\alph*),ref=(\alph*),leftmargin=*,itemsep=\medskipamount}
\setlist[enumerate,2]{label=(\alph{enumi}-\arabic*),ref=(\alph{enumi}-\arabic*),
                      leftmargin=*,topsep=\smallskipamount,itemsep=\smallskipamount}

\newtheorem{theorem}{Theorem}[section]
\newtheorem{maintheorem}{Theorem}

\newtheorem{maintheorem*}{Main Theorem}
\newtheorem{proposition}[theorem]{Proposition}
\newtheorem{lemma}[theorem]{Lemma}
\newtheorem{corollary}[theorem]{Corollary}

\theoremstyle{definition}

\newtheorem{definition}[theorem]{Definition}
\newtheorem{remark}[theorem]{Remark}
\newtheorem*{remark*}{Remark}
\newtheorem{example}[theorem]{Example}

\numberwithin{equation}{section}

\DeclareMathOperator{\Ext}{Ext}

\hypersetup{pdftitle={The classification of taut surface singularities via representations from logarithmic vector fields},pdfsubject={Logarithmic vector fields, exceptional cycles, and taut surface singularities},pdfkeywords={taut singularity, surface singularity, resolution graph, finite-dimensional algebra, logarithmic tangent bundle}}

\begin{document}

\title[Taut surface singularities and representations]{The classification of taut surface singularities via representations from logarithmic vector fields}

\author[H. Park]{Heesang Park}

\address{Department of Mathematics, Konkuk University, Seoul, Republic of Korea}

\email{HeesangPark@konkuk.ac.kr}

\author[D. Shin]{Dongsoo Shin}

\address{Department of Mathematics, Chungnam National University, Daejeon, Republic of Korea}

\email{dsshin@cnu.ac.kr}



\begin{abstract}
We study logarithmic vector fields on exceptional cycles through finite-dimensional representations and recover Laufer's classification of taut normal complex surface singularities. For graphs with two nodes, restriction maps define representations of a fixed quiver, with dimensions bounded independently of the number of components and their self-intersections. We identify restriction cohomology with an extension group and use interval decomposition to reduce its vanishing to interpolation on a rational chain. These calculations determine the classification conditions, and vanishing on every analytic realization proves sufficiency, including the nonrational cases. The obstruction cycles also give a uniform criterion: a graph is taut if and only if \(H^1(5E,S_X|_{5E})=0\) for every analytic realization, where \(S_X\) is the ambient logarithmic tangent bundle. We prove that five is the smallest positive integer for which this criterion holds for all graphs.
\end{abstract}

\maketitle

\section{Introduction}
\label{sec:introduction}

We study the cohomology of logarithmic vector fields on exceptional cycles by means of finite-dimensional representations. Let \((X,E)\to(V,0)\) be the minimal good resolution of a normal complex surface singularity, with \(E\) reduced, and put \(S_X=\Theta_X(-\log E)\). For resolution graphs with two nodes, we construct a positive exceptional cycle \(Z\) and decompose it into pieces on which the first cohomology of the restricted logarithmic tangent bundle vanishes. The restriction maps on their intersections define a representation of a fixed quiver, whose dimension is bounded independently of the number of exceptional components and their self-intersections. We identify \(H^1(Z,S_X|_Z)\) with an extension group over the path algebra of this quiver. Interval decomposition expresses its dimension as a sum of three multiplicities. When a nonempty connecting chain remains, only one multiplicity contributes, and the calculation of the restriction images reduces its vanishing to interpolation on a chain of rational curves.

These calculations give another proof of the classification of taut normal complex surface singularities due to Laufer~~\cite{Laufer1973}. Restriction to the reduced exceptional divisor first bounds the topology of a taut graph. The restrictions along the arms determine the one-node cases, and the representations and interpolation calculations determine the remaining conditions for graphs with two nodes. We prove sufficiency by extending cohomology vanishing from the chosen cycles to every analytic realization, including the nonrational cases. The necessity arguments also produce, for every excluded graph, an analytic realization with a cycle \(Y\leq5E\) such that \(H^1(Y,S_X|_Y)\neq0\). Consequently, a graph is taut if and only if \(H^1(5E,S_X|_{5E})=0\) for every analytic realization. An explicit example shows that five is the smallest positive integer for which this criterion holds for all graphs.

\subsection{Main results}

Let \((V,0)\) be a normal complex surface singularity, and let \(\pi\colon(X,E) \to (V,0)\) be its minimal good resolution, with \(E\) reduced. A dual graph is \emph{taut} if any two singularities with that minimal good resolution graph are analytically isomorphic. We work over \(\mathbb{C}\) throughout the paper.

Put
\[
 S_X = \Theta_X(-\log E),\qquad
 L_X = \bigwedge\nolimits^2 S_X
 = \mathcal{O}_X(-K_X-E).
\]

A vertex of valency at least three is called a \emph{node}. We write the self-intersection attached to a vertex \(v\) as \(-b_v\). For a nonempty chain with self-intersections \(-b_1,\ldots,-b_l\), where \(b_i \geq 2\), put
\[
 [b_1,\ldots,b_l]
 := b_1-\frac{1}{b_2-\displaystyle\frac{1}{\ddots-\displaystyle\frac{1}{b_l}}}
 = \frac{n}{a}.
\]
We use the same bracket notation for the chain and its continued fraction. Arms are read from the node outwards. A chain may have length one, and the chain joining two nodes may be empty.

\begin{maintheorem}[Classification]\label{thm:classification}
A connected negative-definite minimal good resolution graph is taut if and only if every vertex has genus zero and the graph belongs to one of the four classes described below. In the following list, every vertex has weight at most \(-2\), except as stated in~\ref{intro:cycles}.
\end{maintheorem}

\begin{enumerate}
%
\item\label{intro:chains} \emph{Chains.} Every vertex has weight \(\leq -2\).
  \begin{center}
  \begin{tikzpicture}[graph,scale=0.75]
    \coordinate (A) at (0,0);
    \coordinate (B) at (2,0);
    \chainarm{A}{B}
    \vtx{A}\vtx{B}
  \end{tikzpicture}
  \end{center}
%
\item\label{intro:cycles} \emph{Cycles.} All weights are \(\leq -2\), with at least one \(\leq -3\).
  Two vertices joined by two edges may also have weights \(-1,-c\), where \(c\geq5\).
  \begin{center}
  \begin{tikzpicture}[graph,scale=0.75]
    \coordinate (A) at (-0.9,0);
    \coordinate (B) at (0.9,0);
    \coordinate (C) at (0,1.5);
    \chainarm{A}{B}\chainarm{B}{C}\chainarm{C}{A}
    \vtx{A}\vtx{B}\vtx{C}
    \coordinate (D) at (3.6,0.75);
    \coordinate (E) at (5.2,0.75);
    \draw (D) .. controls ($(D)+(0.4,0.55)$) and ($(E)+(-0.4,0.55)$) .. (E);
    \draw (D) .. controls ($(D)+(0.4,-0.55)$) and ($(E)+(-0.4,-0.55)$) .. (E);
    \vtx{D}\vtx{E}
    \node[left=2pt]  at (D) {\(-1\)};
    \node[right=2pt] at (E) {\(-c\)};
  \end{tikzpicture}
  \end{center}
%
\item\label{intro:one-node} \emph{Trees with one node.} The valency of the node is 3 and all weights are \(\leq -2\).
  \begin{center}
  \begin{tikzpicture}[graph,scale=0.75]
    \coordinate (N) at (0,0);
    \coordinate (L) at (-2,0);
    \coordinate (R) at (2,0);
    \coordinate (U) at (0,1.5);
    \chainarm{N}{L}\chainarm{N}{R}\chainarm{N}{U}
    \vtx{N}\vtx{L}\vtx{R}\vtx{U}
  \end{tikzpicture}
  \end{center}

  \begin{enumerate}
  \item If the node has weight \(\leq -3\), the arms are arbitrary.
  \item If the node has weight \(-2\), then exactly one of the following rows applies.
  \end{enumerate}

    \begin{center}
    \begin{tabular}{@{}lll@{}}
    \toprule
    First arm & Second arm & Third arm \\
    \midrule
    \([p]\) \((p\geq2)\) & \(n_2/a_2\geq2\)          & any \\
    \([p]\) \((p\geq3)\) & \([2,m]\) \((m\geq2)\)   & \(n_3/a_3<2\) \\
    \([2]\)              & \([2,m]\) \((m\geq2)\)   & \(4/3\leq n_3/a_3<2\), or \([2,2,2,q]\) \((q\geq2)\) \\
    \bottomrule
    \end{tabular}
    \end{center}
%
\item\label{intro:two-node} \emph{Trees with two nodes.} The valencies of both nodes \(P\) and \(Q\) are 3.
  \begin{center}
  \begin{tikzpicture}[graph,scale=0.75]
    \coordinate (P)  at (0,0);
    \coordinate (Q)  at (2.2,0);
    \coordinate (PL) at (-2,0);
    \coordinate (PU) at (0,1.5);
    \coordinate (QR) at (4.2,0);
    \coordinate (QU) at (2.2,1.5);
    \chainarm{P}{PL}\chainarm{P}{PU}
    \chainarm{Q}{QR}\chainarm{Q}{QU}
    \chainarm{P}{Q}
    \vtx{P}\vtx{Q}\vtx{PL}\vtx{PU}\vtx{QR}\vtx{QU}
    \node[below=3pt] at (P) {\(P\)};
    \node[below=3pt] at (Q) {\(Q\)};
  \end{tikzpicture}
  \end{center}

  Let \(-b_P\) and \(-b_Q\) be the weights of \(P\) and \(Q\), respectively; then \(b_P,b_Q\geq2\).
  The chain connecting \(P\) and \(Q\) has weights \(-c_1,\dotsc,-c_e\), read from \(P\), and we put
  \(d=\sum_{i=1}^{e}(c_i-2)\).
  The integers \(s_P\) and \(s_Q\) count the outward arms of length \(\geq2\) at \(P\) and \(Q\), respectively.

  \begin{enumerate}
  \item If \(b_P\geq3\) and \(b_Q\geq3\), then the four outward arms and the connecting chain are arbitrary.
  \item Assume that \(b_P=2\). The outward arms at \(P\) and the connecting chain must satisfy
    Table~\ref{intro:two-node-table}. In addition, if \(s_P=1\) and \(d=0\), then \(b_Q\geq3\) and
    \(s_Q\leq b_Q-2\). If \(b_Q=2\), the same conditions apply after interchanging \(P\) and \(Q\)
    and reading the connecting chain from \(Q\).
  \end{enumerate}
\end{enumerate}

\begin{table}[htbp]
\centering
\small
\setlength{\tabcolsep}{5pt}
\begin{threeparttable}
\begin{tabular}{@{}c l l >{\raggedright\arraybackslash}p{0.38\linewidth} c@{}}
\toprule
\(s_P\) & First arm & Second arm & Chain \([c_1,\dotsc,c_e]\) & Condition \\
\midrule
\(0\) & \([p]\) & \([q]\) & any & --- \\
\midrule
\(1\) & \([p]\)              & \([a_1,\dotsc,a_n]\) \((a_1\geq3)\)     & any & --- \\
      & \([p]\) \((p\geq3)\) & \([2,m]\)                               & any & --- \\
      & \([2]\)              & \([2,m]\)                               & [\,], or does not begin with \([2,2,\dotsc]\) & --- \\
      & \([2]\)              & \([2,m]\)                               & \([2,2]\), or \([2,2,c_3,\dotsc,c_e]\) \((c_3\geq3)\) & \(b_Q+d\geq4\) \\
      & \([p]\)              & \([2,a_2,\dotsc,a_n]\) \((n\geq3)\)     & \([\,]\), or \([c_1,\dotsc,c_e]\) \((c_1\geq3)\) & \(b_Q+d\geq4\) \\
\bottomrule
\end{tabular}
\begin{tablenotes}[flushleft]\footnotesize
\item \([\,]\) denotes the empty chain; a dash means no condition.
      The last column applies only when \(s_Q\geq1\).
\end{tablenotes}
\end{threeparttable}
\caption{Conditions at a node \(P\) with \(b_P=2\).}
\label{intro:two-node-table}
\end{table}


\begin{remark*}
We require every irreducible component of the exceptional divisor of a good resolution to be smooth. For the two-vertex cycle with weights \(-1,-c\), where \(c \geq 5\), contracting the component of self-intersection \(-1\) produces a nodal rational curve. Thus this contraction does not give a good resolution.
\end{remark*}

\begin{remark*}
The list above coincides with Laufer's list of taut resolution graphs~~\cite[pp.~136--139]{Laufer1973}. The continued-fraction conditions describe the same one-node families, and the two-node conditions are equivalent to Laufer's Table~\textup{(IV)} with its permitted contraction operations. See Section~\ref{cls:laufer-comparison} for the correspondence between the families.
\end{remark*}

The obstruction calculations also give a uniform criterion on a fixed thickening of the reduced exceptional divisor.

\begin{maintheorem}[Theorem~\ref{thick:criterion} and Example~\ref{thick:optimality}]\label{intro:thickening}
A connected negative-definite minimal good resolution graph \(\Gamma\) is taut if and only if
\[
H^1(5E,S_X|_{5E}) = 0
\]
for every analytic realization \((X,E) \to (V,0)\) of \(\Gamma\). The integer \(5\) cannot be replaced by any smaller positive integer.
\end{maintheorem}

The sheaf is the restriction of the ambient logarithmic tangent bundle. Vanishing is required for every analytic realization; the theorem does not give a criterion for a fixed realization.

For a two-node graph satisfying the restrictions stated in Proposition~\ref{rep:local-restrictions}, we choose a positive exceptional cycle \(Z\). It is the union of two cycles near the nodes and, when a nonempty connecting chain remains, twice that reduced chain. Their intersections have coordinate rings \(\mathbb{C}[t,n]/(t^j,n^2)\), with \(j \leq 4\). The images of the section spaces under restriction, together with the restriction maps from the middle cycle, define the representation in the following theorem.

For \(1 \leq i \leq j \leq 5\), let \(I[i,j]\) denote the interval representation of the quiver in the following theorem, with a one-dimensional space at vertices \(i,\ldots,j\), zero elsewhere, and identity maps between its nonzero spaces.

\begin{maintheorem}[Propositions~\ref{rep:ext-comparison}, \ref{thm:three-intervals}, and~\ref{prop:single-end}]
\label{thm:restriction-representation}
Let \(\Lambda\) be the path algebra of \(1 \to 2 \leftarrow 3 \to 4 \leftarrow 5\), and put \(J_0 = I[1,5]\). For every analytic realization, the chosen cycle defines a \(\Lambda\)-module \(M_0\), of dimension at most \(96\), such that
\[
H^1(Z,S_X|_Z) \simeq \operatorname{Ext}^1_{\Lambda}(J_0,M_0).
\]
If \(m_{ij}\) is the multiplicity of \(I[i,j]\) in \(M_0\), then
\[
h^1(Z,S_X|_Z) = m_{22}+m_{24}+m_{44}.
\]
If the part of the connecting chain outside the two cycles near the nodes is nonempty, then \(m_{22} = m_{44} = 0\). If the two cycles intersect and no connecting component remains, only \(m_{22}\) can contribute.
\end{maintheorem}

The quiver and the bound \(96\) are independent of the lengths of the arms and connecting chain and of all vertex weights. The dimension bound follows from the bounded lengths of the intersection schemes used to define \(M_0\).

\subsection{The representation-theoretic calculation}

Laufer~~\cite[Section~2]{Laufer1973} computes tangent sheaf cohomology on infinitesimal neighborhoods constructed by plumbing. His calculations use explicit \v{C}ech cocycles and the continued fractions of the rational chains. We also calculate restrictions of logarithmic vector fields along rational chains. In the two-node case, we express the obstruction to extending compatible sections across an exceptional cycle as an extension group of a finite-dimensional representation. This separates the calculation of the local restriction images from the calculation of their common cokernel.

The cohomological criterion states that a graph is taut if and only if \(H^1(X,S_X)=0\) for every analytic realization; see Wahl~~\cite[Theorem~9.1]{Wahl2016}. We first study this cohomology on a positive exceptional cycle \(Z\). Restriction gives a surjection from \(H^1(X,S_X)\) to \(H^1(Z,S_X|_Z)\), so nonvanishing on \(Z\) excludes tautness. For the admitted graphs, we prove vanishing on \(Z\) for every analytic realization and then extend it to the resolution by exact sequences and the theorem on formal functions.

For a graph with two nodes, we choose cycles \(W_L,W_R\) near the nodes. If a nonempty connecting chain remains outside their supports, let \(K\) be its reduced cycle and put
\[
Z=W_L+2K+W_R,\qquad D_i=W_i\cap2K
\quad(i=L,R).
\]
The cycles are chosen so that the restrictions of \(S_X\) to \(W_L\), \(2K\), and \(W_R\) have vanishing first cohomology. The exact sequence for their union therefore identifies \(H^1(Z,S_X|_Z)\) with the cokernel of the difference of the restriction maps
\[
H^0(W_L,S_X|_{W_L})
\oplus H^0(2K,S_X|_{2K})
\oplus H^0(W_R,S_X|_{W_R})
\to
H^0(D_L,S_X|_{D_L})
\oplus H^0(D_R,S_X|_{D_R}).
\]
Thus the required cohomology measures whether sections on the three pieces can be chosen to agree on both intersection schemes.

The five vertices of the quiver correspond to the two endpoint pieces, the middle piece, and their two intersections. More precisely, put
\[
\begin{aligned}
P_i&=H^0(D_i,S_X|_{D_i}),\\
V_i&=\operatorname{im}\bigl(
H^0(W_i,S_X|_{W_i})\to P_i
\bigr),\\
V_K&=\operatorname{im}\bigl(
H^0(2K,S_X|_{2K})\to P_L\oplus P_R
\bigr).
\end{aligned}
\]
The inclusions and projections define the representation
\[
V_L\to P_L\leftarrow V_K\to P_R\leftarrow V_R.
\]
The two restrictions from the middle piece are retained in the single space \(V_K\); they must arise from the same section on \(2K\). Replacing the section spaces by these images preserves the cokernel. The projective resolution of \(J_0=I[1,5]\) gives the same cokernel, and hence
\[
H^1(Z,S_X|_Z)
\simeq
\operatorname{Ext}^1_{\Lambda}(J_0,M_0).
\]
The intersection schemes have uniformly bounded lengths. Consequently, the dimension of \(M_0\) is bounded independently of the lengths of the arms and connecting chain and of their self-intersections.

Every representation of this quiver is a direct sum of interval representations. Computing the extension group against each interval shows that only three interval types contribute. If \(m_{ij}\) denotes the multiplicity of \(I[i,j]\), then
\[
h^1(Z,S_X|_Z)=m_{22}+m_{24}+m_{44}.
\]
When \(K\neq0\), a section on \(2K\), together with a section on either endpoint piece, can satisfy the restriction condition at that endpoint. These two surjectivity statements give \(m_{22}=m_{44}=0\). The remaining multiplicity \(m_{24}\) measures the obstruction to satisfying both endpoint conditions simultaneously. When \(K=0\), the two endpoint cycles meet directly, and their union sequence gives the corresponding representation supported on three vertices.

We calculate the restriction images in Section~\ref{sec:representations}. Along an arm, the self-intersections determine the propagation of monomial exponents and logarithmic coefficients. These calculations give the allowable constant coefficients and higher-order coefficients at each intersection. On the middle piece, put \(I=\mathcal{O}_K(-K)\). The exact sequence
\[
0\to I\otimes S_X|_K
\to S_X|_{2K}
\to S_X|_K
\to0
\]
separates the restriction to the reduced chain from the terms in the square-zero ideal. Since \(S_X|_K\simeq\mathcal{O}_K^{\oplus2}\), the latter terms are sections of \(I^{\oplus2}\). The remaining cokernel is therefore determined by whether these sections attain the prescribed endpoint jets, modulo the images supplied by \(W_L\) and \(W_R\). This is the interpolation problem treated in Section~\ref{sec:interpolation}. Tridiagonal determinants calculate the transport along the chain and determine the relative positions of the endpoint lines. The resulting vanishing conditions give the two-node restrictions in Theorem~\ref{thm:classification}.

The necessity arguments also control the multiplicities needed to detect a failure of tautness. For every excluded graph, they produce a suitable analytic realization and a cycle \(Y\leq5E\) with \(H^1(Y,S_X|_Y)\neq0\). In particular, Proposition~\ref{int:critical-necessity} proves the last two-node restriction by constructing a cycle with
\[
\chi(Y,S_X|_Y)=-1.
\]
It follows that \(h^1(Y,S_X|_Y)=h^0(Y,S_X|_Y)+1>0\). Restriction from \(5E\) to \(Y\) then gives nonvanishing on \(5E\). Together with vanishing for every realization of an admitted graph, this proves Theorem~\ref{thick:criterion}. Finally, Example~\ref{thick:optimality} gives a non-taut graph for which the restriction cohomology vanishes on \(rE\), for every analytic realization and every \(1\leq r\leq4\). Hence five is the smallest positive integer that gives this criterion for all graphs.

\subsection{Outline of the proof}

Section~\ref{sec:foundations} reduces the possible graphs to chains, cycles, and trees with at most two trivalent nodes, and proves tautness for chains and cycles. Section~\ref{sec:stars} classifies the one-node graphs by restriction to the central curve and the arms, together with the weighted homogeneous calculation and Wahl's vanishing theorem.

Sections~\ref{sec:representations} and~\ref{sec:interpolation} construct the two-node cycles, compute their restriction images, and derive the remaining conditions by interval decomposition and interpolation. Section~\ref{sec:classification} extends vanishing from these cycles to every analytic realization. When \(H^1(Z,\mathcal{O}_Z)=0\), nef twists and formal functions extend the vanishing to \(X\). The exceptional nonrational two-node family is treated by exact sequences on two reduced rational trees. Finally, the nonvanishing cycles used in the necessity proofs can be chosen inside \(5E\), giving Theorem~\ref{intro:thickening}; Example~\ref{thick:optimality} proves optimality.

\subsection*{Statements and Declarations}

\begin{itemize}[labelindent=0pt, leftmargin=*, align=left]
\item \textbf{Funding.} Heesang Park was supported by the Basic Science Research Program through the National Research Foundation of Korea (NRF) (RS-2026-25590534). Dongsoo Shin was supported by the Basic Science Research Program through the National Research Foundation of Korea (NRF) (NRF-2023R1A2C1007293).


\item \textbf{Use of artificial intelligence.} The authors used ChatGPT-6 Astra for assistance with logarithmic restriction and representation calculations, checking proofs, and language editing. The authors take full responsibility for the contents of this paper.
\end{itemize}

\section{Logarithmic restrictions and possible graphs}
\label{fnd:section}\label{sec:foundations}

Let \(\pi\colon (X,E) \to (V,0)\) be the minimal good resolution of a normal complex surface singularity. We choose a sufficiently small Stein representative of \(V\). Thus \(X\) is smooth, \(\pi\) is proper, and its fibers have dimension at most one. Write
\[
E = \sum_{v}E_v,\qquad E_v^2 = -b_v,\qquad g_v = g(E_v).
\]
Every \(E_v\) is smooth, and \(E\) has simple normal crossings. Distinct components may meet more than once. The valency \(t_v\) is the number of intersection points of \(E_v\) with the other components, counted separately. A node is a vertex of valency at least three.

We use the locally free sheaves
\[
S_X = \Theta_X(-\log E),
\qquad
L_X = \det S_X = \mathcal{O}_X(-K_X-E).
\]
If \(E = (y = 0)\), the sheaf \(S_X\) has local basis \(\partial_x,y\partial_y\). At a crossing \(E = (xy = 0)\), it has basis \(x\partial_x,y\partial_y\). An exceptional cycle is an effective integral divisor supported on \(E\). If \(Z\) is such a cycle and \(\mathcal{F}\) is locally free on \(X\), we write \(\mathcal{F}|_Z = \mathcal{F}\otimes_{\mathcal{O}_X}\mathcal{O}_Z\). All restrictions in this paper refer to ambient sheaves in this sense.

\subsection{Restriction to finite cycles}

Restriction to exceptional cycles and the theorem on formal functions relate cohomology on \(X\) to cohomology on finite cycles.

\begin{lemma}
\label{fnd:restriction}
For every coherent sheaf \(\mathcal{G}\) on \(X\), one has \(H^q(X,\mathcal{G}) = 0\) for \(q \geq 2\). Consequently, if \(\mathcal{F}\) is locally free and \(Z\) is an exceptional cycle, restriction induces a surjection
\[
H^1(X,\mathcal{F}) \twoheadrightarrow H^1(Z,\mathcal{F}|_Z).
\]
If \(Y \leq Z\), restriction also induces a surjection
\[
H^1(Z,\mathcal{F}|_Z) \twoheadrightarrow H^1(Y,\mathcal{F}|_Y).
\]
\end{lemma}

\begin{proof}
The coherence of the higher direct images follows from Grauert~~\cite[\S6, Hauptsatz~I, p.~59]{Grauert1960}. The theorem on formal functions of Grauert~~\cite[\S6, Hauptsatz~II~a, p.~59]{Grauert1960} and the dimension of the fibers give \(R^q\pi_*\mathcal{G}=0\) for \(q\geq2\). Leray over the Stein base gives \(H^q(X,\mathcal{G})=0\) in these degrees. Apply this vanishing to the kernel of \(\mathcal{F}\to\mathcal{F}|_Z\). The second surjection follows because the kernel of \(\mathcal{F}|_Z\to\mathcal{F}|_Y\) is supported in dimension at most one.
\end{proof}

For effective exceptional cycles \(A,B\), the exact sequence of ideal sheaves gives
\begin{equation}
\label{fnd:cycle-sequence}
0 \to \mathcal{F}(-A)|_B
 \to \mathcal{F}|_{A+B}
 \to \mathcal{F}|_A
 \to 0.
\end{equation}
Its first term is obtained from \(\mathcal{O}_X(-A)/\mathcal{O}_X(-A-B)\), so the sequence also holds when the cycles have common components.

\begin{lemma}
\label{fnd:formal-vanishing}
Let \(Z\) be an exceptional cycle whose support is \(E\), and let \(\mathcal{F}\) be locally free on \(X\). If
\[
H^1\left(Z,\mathcal{F}(-kZ)|_Z\right) = 0
\quad\text{for all }k \geq 0,
\]
then \(H^1(X,\mathcal{F}) = 0\).
\end{lemma}

\begin{proof}
Apply the cycle sequence with \(A=kZ\) and \(B=Z\). Its kernel is \(\mathcal{F}(-kZ)|_Z\), so induction gives \(H^1(kZ,\mathcal{F}|_{kZ})=0\) for every \(k\geq1\). The ideals of \(Z\) and of the exceptional fiber have the same radical, so their powers define cofinal systems of neighborhoods of \(E\). The theorem on formal functions of Grauert~~\cite[\S6, Hauptsatz~II~a, p.~59]{Grauert1960} therefore identifies the completion of \((R^1\pi_*\mathcal{F})_0\) with zero. Coherence implies that this stalk vanishes. Since \(\pi\) is an isomorphism away from \(0\), it follows that \(R^1\pi_*\mathcal{F}=0\), and Leray over the Stein base gives the assertion.
\end{proof}

\subsection{Component restrictions and logarithmic fibers}

We use the component restrictions and the ordered logarithmic fibers described by Wahl~~\cite[\S2 and Lemma~2.1]{Wahl2016}. Let \(D_v\) be the reduced intersection divisor on \(E_v\).

\begin{lemma}[{Wahl~~\cite[\S2 and Lemma~2.1]{Wahl2016}}]
\label{fnd:component}\label{fnd:logarithmic-fiber}
There is an exact sequence
\begin{equation}
\label{fnd:component-sequence}
0 \to \mathcal{O}_{E_v}\to S_X|_{E_v}\to\Theta_{E_v}(-D_v)\to0.
\end{equation}
It splits if \(t_v>0\). If also \(E_v\simeq\mathbb{P}^1\), then
\[
S_X|_{E_v}\simeq\mathcal{O}_{\mathbb{P}^1}\oplus\mathcal{O}_{\mathbb{P}^1}(2-t_v).
\]
The first map sends \(1\) to the normal Euler vector field. At a crossing \(E=(xy=0)\), the classes of \(x\partial_x,y\partial_y\) form an ordered basis of the logarithmic fiber independent of the equations for the two ordered branches.
\end{lemma}

\begin{lemma}
\label{fnd:two-point-transport}
Let \(C\simeq\mathbb{P}^1\) have self-intersection \(-b\) and meet the remaining divisor at exactly two points \(p,q\). In logarithmic coefficients ordered tangent first and normal second, transport from \(p\) to \(q\), and then exchange of the entries at the next crossing, are respectively
\[
A_b=\begin{pmatrix}-1&0\\b&1\end{pmatrix},
\qquad
T_b=\begin{pmatrix}b&1\\-1&0\end{pmatrix}.
\]
\end{lemma}

\begin{proof}
As in Wahl~~\cite[proof of Lemma~2.5]{Wahl2016}, coordinates on the first infinitesimal neighborhood preserving the two marked branches have transition \(z'=z^{-1},y'=z^by\). The obstruction to this normalization lies in \(H^1(C,\Theta_C(-p-q)\otimes N_{C/X}^{\vee})=H^1(\mathbb{P}^1,\mathcal{O}(b))=0\). Thus
\[
z\partial_z=-z'\partial_{z'}+by'\partial_{y'},\qquad y\partial_y=y'\partial_{y'}.
\]
These identities give \(A_b\); exchanging the two entries gives \(T_b\). The logarithmic fibers are independent of the chosen branch equations by Lemma~\ref{fnd:logarithmic-fiber}.
\end{proof}

\subsection{Tautness and restriction after forgetting components}

The following criterion characterizes tautness by logarithmic cohomology. For the formulation using a sufficiently thick plumbing scheme, see Laufer~~\cite[Theorems~3.9--3.10]{LauferDeformations}.

\begin{theorem}[{Wahl~~\cite[Theorem~9.1]{Wahl2016}}]
\label{fnd:tautness}
A graph \(\Gamma\) is taut if and only if
\[
H^1(X,S_X) = 0
\]
for every analytic realization \((X,E) \to (V,0)\) of \(\Gamma\).
\end{theorem}

Thus one realization with nonzero first cohomology disproves tautness, whereas the sufficient direction requires vanishing on every realization.

\begin{lemma}
\label{fnd:forget-branches}
Let \(D\) be a reduced union of components of \(E\), and put
\[
S_D = \Theta_X(-\log D).
\]
If \(Y\) is an effective cycle supported on \(D\), there is an exact sequence
\[
0 \to S_X|_Y \to S_D|_Y \to Q_Y \to 0,
\]
where \(Q_Y\) has zero-dimensional support. Consequently,
\[
H^1(X,S_X) \twoheadrightarrow 
H^1(Y,S_X|_Y) \twoheadrightarrow 
H^1(Y,S_D|_Y).
\]
\end{lemma}

\begin{proof}
The sheaves differ only where \(D\) meets an omitted component. In coordinates \(D=(y=0)\), \(E-D=(x=0)\), their quotient is \(S_D/S_X=\mathcal{O}_{(x=0)}\partial_x\). If \(Y=(y^n=0)\), multiplication by \(y^n\) is injective on this quotient. Restriction therefore preserves exactness, and its quotient has finite length. The cohomology surjections follow from \(H^1(Q_Y)=0\) and Lemma~\ref{fnd:restriction}.
\end{proof}

\subsection{Analytic realization of subgraphs}

We shall use analytic plumbing only to construct a realization with a prescribed obstruction. We do not assume that every realization of a graph contains a fixed analytic realization of a smaller graph.

\begin{lemma}
\label{fnd:plumbing}
Let \(D\) be a compact connected normal-crossings divisor on a smooth complex surface. Suppose that the weighted graph of \(D\) is a full subgraph of a negative-definite weighted graph \(\Gamma\). Assume that the components to be added are smooth rational curves, and that each new edge incident with \(D\) is assigned a distinct smooth point of \(D\). After shrinking the neighborhood of \(D\), there is a smooth complex surface containing that neighborhood and a normal-crossings divisor with graph \(\Gamma\). The resulting divisor can be contracted to a normal surface singularity.
\end{lemma}

\begin{proof}
At a prescribed attachment choose coordinates \(D=(y=0)\). Complete the transverse disk \(x=0\) to a rational curve with normal degree \(-b\) by the transition \(s'=s^{-1},u'=s^bu\), taking \(s=y,u=x\) on the initial disk. This preserves the original charts near \(D\). For further attachments use distinct marked disks and interchange tangent and normal coordinates. The prescribed normal bundles fix all self-intersections. The plumbing construction of Laufer~~\cite[Section~2]{Laufer1973}, with nested coordinate neighborhoods and sufficiently small transverse disks, gives a Hausdorff surface containing a neighborhood of the original divisor. Its full intersection matrix is negative definite, so Grauert~~\cite{Grauert1962} gives a contraction of the divisor.
\end{proof}

In our applications every vertex of the subgraph has weight at most \(-2\), so the subgraph is the graph of a minimal good resolution.

\subsection{Nef twists on finite cycles}

The following elementary argument uses the divisor method in the proof of Lipman~~\cite[Proposition~11.1(i)--(ii)]{Lipman1969}.

\begin{lemma}
\label{fnd:nef-twists}\label{cls:nef}
Let \(C\) be a positive exceptional cycle whose reduced divisor is a tree of rational curves. Suppose \(H^1(C,\mathcal{O}_C) = 0\). If \(\mathcal{F}\) is locally free on \(C\), and a line bundle \(\mathcal{A}\) has nonnegative degree on every irreducible component, then \(H^1(C,\mathcal{F}) = 0\) implies that \(H^1(C,\mathcal{F}\otimes\mathcal{A}) = 0\).
\end{lemma}

\begin{proof}
Let \(\mathcal{I}\) be the nilradical. Constants on \(C_{\mathrm{red}}\) lift to \(C\), so the nilradical sequence gives \(H^1(C,\mathcal{I})=0\); also \(H^2(C,\mathcal{I})=0\). The finite logarithm identifies \(1+\mathcal{I}\) with \(\mathcal{I}\), and the unit sequence gives \(\operatorname{Pic}(C)\simeq\operatorname{Pic}(C_{\mathrm{red}})\). Since the reduced divisor is a rational tree, these groups are determined by the component degrees.

Choose an effective Cartier divisor \(B\), supported at smooth points of the reduced components, with these degrees. A tangent parameter is a nonzerodivisor on \(C\), so these divisors exist also when \(C\) is nonreduced. The degree description gives \(\mathcal{A}\simeq\mathcal{O}_C(B)\), and multiplication by its section gives
\[
0\to\mathcal{F}\to\mathcal{F}\otimes\mathcal{A}\to Q\to0
\]
with \(Q\) of finite length. Hence \(H^1(C,\mathcal{F})\to H^1(C,\mathcal{F}\otimes\mathcal{A})\) is surjective.
\end{proof}

\subsection{The underlying graph}
\label{red:section}\label{sec:reduction}

The restriction calculation of Wahl~~\cite[Lemmas~2.3--2.5 and Proposition~2.6]{Wahl2016} bounds the number of branching vertices without using a classification of taut graphs.

\begin{proposition}\label{red:component-obstructions}
If the graph of \(E\) is taut, every component is rational and has valency at most three.
\end{proposition}

\begin{proof}
Restriction and Lemma~\ref{fnd:component} make \(H^1(E_v,\Theta_{E_v}(-D_v))\) a quotient of \(H^1(X,S_X)\). The former group is nonzero for positive genus, and for a rational component it has dimension \(\max(t_v-3,0)\). Theorem~\ref{fnd:tautness} proves the assertion.
\end{proof}

\begin{proposition}[{Wahl~~\cite[Proposition~2.6]{Wahl2016}}]\label{red:reduced-dimension}\label{red:sections}
Suppose every component is rational, all valencies are at most three, and there are \(n\geq1\) trivalent vertices. Let \(c\) be the first Betti number of the graph, and put \(\epsilon=1\) for a one-node tree and \(\epsilon=0\) otherwise. Then
\[
h^1(X,S_X)\geq h^1(E,S_X|_E)=n+2c-2+\epsilon.
\]
\end{proposition}

\begin{proof}
Write \(r\) for the number of leaves. Wahl's formula is \(h^1(E,S_X|_E)=r+4c-4+\epsilon\); substituting \(r=n+2-2c\) gives the assertion. The restriction surjection gives the inequality. The proof of Wahl~~\cite[Lemma~2.5]{Wahl2016} uses the nonsingularity of the intersection matrix of a connecting chain, so it also applies to a component of self-intersection \(-1\) in a two-edge cycle.
\end{proof}

\begin{theorem}\label{red:reduction}
Every nonempty taut minimal good resolution graph is a chain or a cycle of rational curves, or a rational tree with one or two trivalent vertices. A cycle may consist of two components meeting twice.
\end{theorem}

\begin{proof}
Proposition~\ref{red:component-obstructions} leaves rational components of valency at most three. Without a trivalent vertex, connectedness gives a chain or cycle. Otherwise Proposition~\ref{red:reduced-dimension} excludes \(c\geq1\) and, for a tree, \(n\geq3\). Minimality then makes every nonbranching component of either remaining tree have self-intersection at most \(-2\).
\end{proof}

\subsection{Chains and cycles}

Component restrictions prove tautness for chains and cycles in every analytic realization. For cycles, negative definiteness makes the reduced logarithmic cohomology vanish.

\begin{proposition}\label{red:chains}
Every negative-definite minimal good resolution graph that is a chain of rational curves is rational and taut. Every self-intersection is at most \(-2\).
\end{proposition}

\begin{proof}
Minimality excludes a component of self-intersection \(-1\). If the chain has at least two components, \(S_X(-mE)|_E\), for \(m\geq0\), has nonnegative summand degrees: they are \(m(b-1),m(b-1)+1\) on an end component of self-intersection \(-b\), and \(m(b-2),m(b-2)\) on an internal component. Its Serre dual has strictly negative summand degrees on the ends and nonpositive degrees internally. A section vanishes on the ends and then successively on all internal components. Thus \(H^1(E,S_X(-mE)|_E)=0\).

For a single component \(C\), the extension in \eqref{fnd:component-sequence} has class \(c_1(N_{C/X})\neq0\), so \(S_X|_C\simeq\mathcal{O}_{\mathbb{P}^1}(1)^{\oplus2}\). Again all the groups \(H^1(C,S_X(-mC)|_C)\) vanish. Lemma~\ref{fnd:formal-vanishing} gives \(H^1(X,S_X)=0\) in every realization, proving tautness. Finally \(\mathcal{O}_E(-mE)\) has nonnegative component degrees; normalization on the rational chain gives its \(H^1\)-vanishing. Formal functions, or Lemma~\ref{fnd:formal-vanishing}, proves rationality.
\end{proof}

\begin{lemma}\label{red:cycle-restriction}
For a negative-definite cycle \(E\) of smooth rational curves, one has \(H^0(E,S_X|_E)=H^1(E,S_X|_E)=0\).
\end{lemma}

\begin{proof}
Write \(E=C_1+\cdots+C_r\), with the components ordered cyclically, and put \(C_i^2=-b_i\). On \(C_i\), order the logarithmic coefficients at the incoming crossing as tangent first and normal second, and denote them by \((a_i,n_i)\). Lemma~\ref{fnd:two-point-transport} shows that compatibility at the outgoing crossing is equivalent to
\[
a_{i+1}=b_i a_i+n_i,
\qquad
n_{i+1}=-a_i.
\]
Eliminating the normal coefficients gives
\[
-b_i a_i+a_{i-1}+a_{i+1}=0.
\]
For two components meeting twice, the two neighboring terms are both \(a_j\), so their sum is \(2a_j\). These equations are precisely the kernel equations for the intersection matrix of \(E\). Negative definiteness gives \(a_i=0\) for every \(i\), and the matching equations then give \(n_i=0\). Hence \(H^0(E,S_X|_E)=0\).

Each component restriction is trivial of rank two. The normalization sequence therefore gives \(\chi(E,S_X|_E)=2r-2r=0\), and consequently \(H^1(E,S_X|_E)=0\).
\end{proof}

\begin{proposition}\label{red:cycles}\label{red:nodal-convention}
Every negative-definite minimal good resolution graph that is a cycle of rational curves is taut. With at least three components its weights are \(-b_i\), where all \(b_i\geq2\) and some \(b_i\geq3\). For two components meeting twice, either \(b_1,b_2\geq2\) and \(b_1b_2>4\), or, after interchange, \(b_1=1\) and \(b_2\geq5\).
\end{proposition}

\begin{proof}
With at least three components, a component of self-intersection \(-1\) can be contracted while keeping smooth components with normal crossings. Minimality therefore gives \(b_i\geq2\), and negative definiteness excludes all \(b_i=2\). With two components, negative definiteness is equivalent to \(b_1,b_2>0\) and \(b_1b_2>4\). Contracting a component of self-intersection \(-1\) makes the other component nodal, so this contraction is not a good resolution.

Suppose all \(b_i\geq2\). The dualizing sheaf of the reduced cycle is trivial: differentials with opposite residues at the two marked points on each normalization component glue to a nowhere-zero section. For \(m\geq1\), the dual of \(S_X(-mE)|_E\) has degrees \(-m(b_i-2)\) in both summands on each component. Its global sections vanish first where \(b_i>2\), then on the degree-zero components by connectedness. Serre duality gives \(H^1(E,S_X(-mE)|_E)=0\). Together with Lemma~\ref{red:cycle-restriction} and formal vanishing, this proves \(H^1(X,S_X)=0\).

In the remaining case write \(E=C+D\), where \(C^2=-1\), \(D^2=-b\), \(CD=2\), and \(b\geq5\). Put \(Z=2C+D\). Then \(ZC=0\), \(ZD=4-b<0\), and \(EC=1\). For \(m\geq0\), set \(\mathcal{F}_m=S_X(-mZ)\). The cycle sequence for \(Z=E+C\) is
\[
0\to\mathcal{F}_m(-E)|_C
\to\mathcal{F}_m|_Z
\to\mathcal{F}_m|_E
\to0.
\]
Its first term is \(\mathcal{O}_C(-1)^{\oplus2}\), independently of \(m\). For \(m=0\), the last term has vanishing first cohomology by Lemma~\ref{red:cycle-restriction}. For \(m\geq1\), its Serre dual restricts to \(\mathcal{O}_C^{\oplus2}\) on \(C\) and to \(\mathcal{O}_D(-m(b-4))^{\oplus2}\) on \(D\). A global section therefore vanishes on \(D\), and hence also on \(C\). Thus \(H^1(E,\mathcal{F}_m|_E)=0\) for every \(m\geq0\). The exact sequence gives
\[
H^1(Z,S_X(-mZ)|_Z)=0
\qquad (m\geq0).
\]
Lemma~\ref{fnd:formal-vanishing} gives \(H^1(X,S_X)=0\). All these arguments apply to every analytic realization, and Theorem~\ref{fnd:tautness} proves tautness.
\end{proof}

\section{Graphs with one node}\label{star:section}\label{sec:stars}

Let \(\Gamma\) be a negative-definite tree with one trivalent node, all of whose vertices have genus zero. Its central curve has self-intersection \(-b\), and its three arms have Hirzebruch--Jung expansions
\[
\frac{n_i}{a_i} = [b_{i1},\ldots,b_{is_i}],
\qquad
0 < a_i < n_i,\qquad
\gcd(a_i,n_i) = 1,\qquad
b_{ij} \geq 2.
\]
The arms are read from the central curve towards their ends. Set
\[
q_k = bk-\sum_{i = 1}^3\left\lceil\frac{ka_i}{n_i}\right\rceil,\qquad
c_k = \#\{i:ka_i \equiv 1\pmod{n_i}\}.
\]
Negative definiteness is equivalent to \(b-\sum_i a_i/n_i > 0\).

\begin{definition}\label{star:standard-plumbing}
Assign the arms to \(0,1,\infty\) on the central curve, whose normal bundle has transition \((z',x')=(z^{-1},z^bx)\). At the attachments use \((t_p,n_p)=(z,x),(z-1,x),(z',x')\), respectively. On an arm component of weight \(-b_j\), use the transition and plumbing identifications
\[
(T'_j,X'_j)=(T_j^{-1},T_j^{b_j}X_j),\qquad
(T_1,X_1)=(n_p,t_p),\qquad
(T_{j+1},X_{j+1})=(X'_j,T'_j).
\]
The sufficiently small surface neighborhood obtained by Lemma~\ref{fnd:plumbing} is called the \emph{standard plumbing realization} of \(\Gamma\). Grauert's theorem contracts its reduced divisor \(D\); if \(\Gamma\) is minimal good, this is the minimal good resolution of the contraction.
\end{definition}

\begin{lemma}\label{star:plumbing-automorphisms}
Let \((X,D)\) and \((\widehat{X},\widehat{D})\) be standard plumbing realizations of the same weighted graph, possibly with different assignments of the arms to \(0,1,\infty\). Every automorphism \(\sigma\) of the weighted graph is induced by a biholomorphism between neighborhoods of \(D\) and \(\widehat{D}\).
\end{lemma}

\begin{proof}
The automorphism fixes the node and permutes arms with identical weight sequences. The induced permutation of their attachment points is realized by a M\"obius transformation \(\mu\). Its linear lift to \(\mathcal{O}_{\mathbb{P}^1}(-1)\), tensored \(b\) times, gives a biholomorphism of the central normal bundle. Near an attachment point it has the form
\[
(t,n) \mapsto (t\alpha(t),n\beta(t)),\qquad \alpha(0)\beta(0) \neq 0.
\]
Identify the source and target arms by their plumbing coordinates. The charts of an arm form a smooth toric surface. The monomial \(t = X_1\) has orders \(r_0 = 0,r_1 = 1\) along the initial boundary and first arm component, with
\[
r_{j+1} = b_jr_j-r_{j-1}.
\]
Since \(b_j \geq 2\), induction gives \(r_{j+1} \geq r_j+1\). Thus \(t\) is holomorphic near the entire arm and vanishes on it.

The torus action extends \((\lambda,\nu)\cdot(T_1,X_1) = (\lambda T_1,\nu X_1)\). On a sufficiently small neighborhood of the arm define
\[
\Phi(P) = (\beta(t(P)),\alpha(t(P)))\cdot P.
\]
This is holomorphic and agrees with the central lift on the overlap. Put \(\varphi(s) = s\alpha(s)\). Since \(t(\Phi(P)) = \varphi(t(P))\) and \(\varphi'(0) \neq 0\), its inverse is
\[
Q \mapsto (\beta(s),\alpha(s))^{-1}\cdot Q,\qquad s = \varphi^{-1}(t(Q)).
\]
Each arm component is torus invariant, so these maps induce the required component permutation. They glue with the central lift, and their inverses glue as well.
\end{proof}

\subsection{Restriction along an arm}

We work in a standard plumbing neighborhood (Definition~\ref{star:standard-plumbing}). Let \(P\) be a central component, and let \(C_1,\ldots,C_s\) be an arm, read away from \(P\), with self-intersections \(-b_i\), where \(b_i \geq 2\). The last component is an end component of the full exceptional divisor, and the other arm components have valency two. Write \(n/a = [b_1,\ldots,b_s]\). Near \(P\cap C_1\), let \(x = 0\) define \(P\), and let \(t = 0\) define \(C_1\). A local frame for \(S_X\) at this crossing is \((x\partial_x,t\partial_t)\).

\begin{lemma}\label{star:vector-arm}
Fix nonnegative integers \(k,j\). The coefficient vectors \((A,B)\) for which
\[
x^kt^j(Ax\partial_x+Bt\partial_t)
\]
extends to a logarithmic vector field on a neighborhood of the whole arm form the subspace
\[
U_{k,j} = \begin{cases}
0,&nj-ak \leq -2,\\
\{(A,B):nA+aB = 0\},&nj-ak = -1,\\
\mathbb{C}^2,&nj-ak \geq 0.
\end{cases}
\]
Put \(A_{\mathrm{arm}}=C_1+\cdots+C_s\). In the chosen frame, the completed restriction of a section of \(S_X|_{(k+1)A_{\mathrm{arm}}}\) at \(P\cap C_1\) belongs to \(\left(\mathbb{C}[[x,t]]/(t^{k+1})\right)^{\oplus2}\). For \(0\leq j\leq k\), the image of the map from \(H^0((k+1)A_{\mathrm{arm}},S_X|_{(k+1)A_{\mathrm{arm}}})\) extracting the coefficient of \(x^kt^j\) is precisely \(U_{k,j}\).
\end{lemma}

\begin{proof}
The exponents satisfy
\[
r_0=k,\qquad r_1=j,\qquad r_{i+1}=b_ir_i-r_{i-1},\qquad r_{s+1}=nj-ak.
\]
The coefficient vector is multiplied by \(T_{b_i}\) at each transition. The first row of \(T_{b_s}\cdots T_{b_1}\) is \((n,a)\), so its terminal normal coefficient is \(nA+aB\). Interior regularity requires nonnegative exponents. At the terminal chart, with exceptional component \(y=0\) and outer divisor \(q=0\), the ambient logarithmic frame is \((y\partial_y,\partial_q)\). An outer exponent \(-1\) therefore permits precisely the line \(nA+aB=0\), whereas an exponent at most \(-2\) permits no nonzero vector. If a compact exponent first becomes negative, the recurrence makes all subsequent exponents strictly decrease. Hence an outer exponent at least \(-1\) ensures nonnegative compact exponents, proving the first assertion.

For the cycle assertion, suppose \(r_{s+1}\leq-1\). Every nonnegative exponent before the first negative exponent is at most \(k\): any strict increase between nonnegative exponents persists by the recurrence and would force a positive outer exponent. Restriction to \((k+1)A_{\mathrm{arm}}\) therefore does not remove any compatibility condition before the first failure of regularity. The invertible transition matrices transport the coefficient vector to the first nonregular chart, where a negative compact exponent or an outer exponent at most \(-2\) forces vanishing; an outer exponent \(-1\) imposes exactly the same terminal line. Distinct characters cannot cancel, since transitions and cycle ideals are monomial. Conversely, a section extending along the whole arm restricts to the cycle. If \(r_{s+1}\geq0\), every coefficient vector extends and therefore restricts.
\end{proof}

Put
\[
d_i = \left\lceil\frac{a_i k-1}{n_i}\right\rceil,\qquad
\epsilon_i = \begin{cases}1,&a_i k \equiv 1\pmod{n_i},\\0,&\text{otherwise.}\end{cases}
\]
One has \(0 \leq d_i \leq k\). The lemma says that both logarithmic coefficient functions must vanish to order \(d_i\) at the attachment. If \(\epsilon_i = 1\), their coefficient of order \(d_i\) must also belong to the indicated one-dimensional subspace. No further condition occurs. In particular the continued fraction and the subspace condition have been obtained from the component restriction maps.

\subsection{The central restriction bundle}

Suppose now that \(P^2 = -b\), that \(P\) has three arms, and that the attachment points are \(0,1,\infty\). In this subsection we index the arms by their attachment points. Set
\[
c = \epsilon_0+\epsilon_1+\epsilon_\infty,\qquad
m = bk-d_0-d_1-d_\infty,\qquad
e = b-\frac{a_0}{n_0}-\frac{a_1}{n_1}-\frac{a_\infty}{n_\infty}.
\]
Let \(\mathcal{F}_k\) be the rank-two bundle of coefficients of degree \(k\) with respect to the normal coordinate of \(P\) in logarithmic vector fields satisfying the arm restrictions of Lemma~\ref{star:vector-arm}.

\begin{proposition}\label{star:central-bundle}
If \(c \leq 2\), then
\[
\mathcal{F}_k \simeq \mathcal{O}_{\mathbb{P}^1}(m-c)\oplus\mathcal{O}_{\mathbb{P}^1}(m-1).
\]
\end{proposition}

\begin{proof}
Before imposing the arm conditions, the line subbundle generated by the normal Euler vector field is \(\mathcal{O}(bk)\), and the quotient is \(\mathcal{O}(bk-1)\). Imposing vanishing of both coefficients to order \(d_i\) at each attachment \(p_i\) tensors these line bundles by \(\mathcal{O}(-\sum_i d_i p_i)\). Each remaining condition \(n_iA+a_iB = 0\) defines a line that projects isomorphically to the tangent direction. It therefore replaces the Euler subbundle by its twist by \(\mathcal{O}(-p_i)\) and leaves the quotient unchanged. We obtain
\[
0 \to \mathcal{O}(m-c) \to \mathcal{F}_k \to \mathcal{O}(m-1) \to 0.
\]
The extension group is \(H^1(\mathbb{P}^1,\mathcal{O}(1-c)) = 0\) when \(c \leq 2\).
\end{proof}

\begin{corollary}\label{star:central-obstruction}
Let \(D\) be the reduced star in its standard plumbing realization. If \(c \leq 2\), there is a natural injection
\[
H^1(\mathbb{P}^1,\mathcal{F}_k) \hookrightarrow H^1((k+1)D,S_X|_{(k+1)D}),
\]
and
\[
h^1(\mathcal{F}_k) = \max(c-m-1,0)+\max(-m,0).
\]
\end{corollary}

\begin{proof}
Let \(N=N_{P/X}\), and put \(\mathcal{G}_k=S_X|_P\otimes N^{-k}\) and \(\mathcal{Q}_k=\mathcal{G}_k/\mathcal{F}_k\). The quotient \(\mathcal{Q}_k\) is supported at the three attachment points. Since \(\mathcal{G}_k\simeq\mathcal{O}(bk)\oplus\mathcal{O}(bk-1)\), one has \(H^1(P,\mathcal{G}_k)=0\). The sequence \(0\to\mathcal{F}_k\to\mathcal{G}_k\to\mathcal{Q}_k\to0\) gives
\[
H^1(P,\mathcal{F}_k)\simeq\operatorname{coker}\left(H^0(P,\mathcal{G}_k)\to H^0(P,\mathcal{Q}_k)\right).
\]
Decompose \((k+1)D\) into \((k+1)P\) and the three disjoint arm cycles. Scaling the normal coordinate of \(P\) extends over the arms through the monomial plumbing transitions, and induces an action on these cycles preserving the restriction maps. In degree \(k\), Lemma~\ref{star:vector-arm} identifies \(H^0(P,\mathcal{Q}_k)\) with the section spaces on the intersections modulo the images from the arm cycles. Indeed, \(0\leq d_i\leq k\), so all jet conditions defining \(\mathcal{Q}_k\) occur on these intersections. The remaining section map is the displayed map from \(H^0(P,\mathcal{G}_k)\). Its cokernel injects into \(H^1((k+1)D,S_X|_{(k+1)D})\) by Mayer--Vietoris. Proposition~\ref{star:central-bundle} gives the dimension formula.
\end{proof}

Since \(d_i = \lceil ka_i/n_i\rceil-\epsilon_i\), one has \(m = q_k+c_k\). If \(c_k = 3\), then \(m = ke+\sum_i1/n_i > 0\). Thus \(q_k+c_k < 0\) forces \(c_k \leq 2\), and Corollary~\ref{star:central-obstruction} gives a nonzero obstruction on \((k+1)D\). Only this standard plumbing realization is needed to exclude tautness. The calculation of the determinant line below proves rationality and sufficiency.

\subsection{Functions and the determinant line}

The construction of Pinkham~~\cite{Pinkham1977}, as described by Wahl~~\cite[proofs of Proposition~3.1 and Theorem~3.3]{Wahl2015}, gives the weighted homogeneous model from distinct points \(p_1,p_2,p_3\) and an integral divisor \(B\) of degree \(b\). Put
\[
\mathcal{D} = B-\sum_{i = 1}^3\frac{a_i}{n_i}p_i,\qquad
\mathcal{A} = \bigoplus_{k \geq 0}\mathcal{O}_{\mathbb{P}^1}(\lfloor k\mathcal{D}\rfloor).
\]
The partial resolution \(Y_{\mathrm{w}}=\operatorname{Spec}_{\mathbb{P}^1}\mathcal{A}\) maps to \(V_{\mathrm{w}}=\operatorname{Spec}H^0(\mathbb{P}^1,\mathcal{A})\). Resolving its three cyclic quotient points gives the resolution with graph \(\Gamma\). One has \(\deg\mathcal{D}>0\) and \(\deg\lfloor k\mathcal{D}\rfloor=q_k\).

\begin{lemma}\label{star:plumbing-model}
Take \(p_1=0\), \(p_2=1\), \(p_3=\infty\), and \(B=b[\infty]\) in Pinkham's construction. A neighborhood of the exceptional divisor in its minimal good resolution is isomorphic to the standard plumbing realization. Under this isomorphism, the grading agrees with scaling the normal coordinate of the central component.
\end{lemma}

\begin{proof}
Consider an arm with \(n/a=[b_1,\dotsc,b_s]\). In the lattice dual to the characters \((t,x)\) at its attachment, the plumbing charts have rays
\[
v_0=(0,1),\qquad v_1=(1,0),\qquad
v_{j+1}=b_jv_j-v_{j-1}\quad(1\leq j\leq s).
\]
The continued-fraction recurrence gives \(v_{s+1}=(n,-a)\), and each adjacent pair has determinant \(-1\). Contracting the arm gives the cyclic quotient chart with invariant characters
\[
t=u^n,\qquad x=u^{-a}v,
\]
for the action \((u,v)\mapsto(\zeta u,\zeta^a v)\), where \(\zeta^n=1\). These characters form a basis of the invariant character lattice. For \(k\geq0\), a function \(g(t)x^k\) is regular on the quotient chart precisely when
\[
\operatorname{ord}_0(g)\geq\left\lceil\frac{ka}{n}\right\rceil.
\]
Together with the central transition \(x'=z^b x\), these conditions identify the contracted plumbing with a neighborhood of the zero section in \(\operatorname{Spec}_{\mathbb{P}^1}\mathcal{A}\). Resolving its quotient charts recovers the original plumbing charts. The character \(x\) has degree one in both constructions, so this identification preserves the grading.
\end{proof}

\begin{lemma}\label{star:local-direct-images}
Let \(f:\widetilde U\to U\) be the minimal resolution of one of the cyclic quotient charts in Pinkham's partial resolution. Let \(B\) be the reduced union of the strict transform of the central curve and the exceptional chain, and put \(M=\mathcal{O}_{\widetilde U}(-(K_{\widetilde U}+B))\). Then
\[
R^1f_*\mathcal{O}_{\widetilde U}=R^1f_*M=0.
\]
\end{lemma}

\begin{proof}
Write the toric boundary of \(\widetilde U\) as \(D_0,\ldots,D_{s+1}\), where \(D_0\) is the strict transform of the central curve, \(D_1,\ldots,D_s\) form the exceptional chain, and \(D_{s+1}\) is the noncompact outer divisor. Since \(B=D_0+\cdots+D_s\), the toric canonical divisor gives \(M\simeq\mathcal{O}_{\widetilde U}(D_{s+1})\). For a fixed character, write \(r_i\) for its order along \(D_i\). These orders satisfy
\[
r_{i+1}=b_i r_i-r_{i-1},\qquad b_i\geq2.
\]
The regularity conditions for functions are \(r_i\geq0\) on every ray. For \(M\), only the last condition changes, to \(r_{s+1}\geq-1\).

Cover \(\widetilde U\) by its successive affine toric charts. A character contributes a section on a chart if the conditions along both boundary divisors hold; on the intersection with the preceding charts, only the condition along the common divisor is required. The resulting Mayer--Vietoris sequences identify the dimension of its first cohomology with one less than the number of connected components of the set of ray indices where regularity fails, or with zero if this set is empty.

For functions, that set is an interval. Indeed, once a negative order is followed by a nonnegative order, the subsequent orders are positive and increasing; once a nonnegative order is followed by a negative order, the subsequent orders are negative and decreasing. Relaxing the condition at the last ray to \(r_{s+1}\geq-1\) cannot disconnect the set of failures. Thus first cohomology vanishes for every character of either sheaf. Since \(U\) is affine, this proves both higher-direct-image vanishings.
\end{proof}

\begin{proposition}\label{star:graded}
Let \((X_{\mathrm{w}},E_{\mathrm{w}})\) be the minimal good resolution of the weighted homogeneous realization of \(\Gamma\). With the grading indexed by the normal coordinate of the central curve, one has
\[
H^1(X_{\mathrm{w}},\mathcal{O}_{X_{\mathrm{w}}})_k
 \simeq H^1(\mathbb{P}^1,\mathcal{O}_{\mathbb{P}^1}(q_k)),
\qquad
H^1(X_{\mathrm{w}},L_{X_{\mathrm{w}}})_k
 \simeq H^1(\mathbb{P}^1,\mathcal{O}_{\mathbb{P}^1}(q_k-1+c_k))
\]
for \(k \geq 0\), and the negative-degree groups vanish. The Euler derivation gives an exact sequence
\begin{equation}\label{star:euler-sequence}
0 \to \mathcal{O}_{X_{\mathrm{w}}} \to S_{X_{\mathrm{w}}} \to L_{X_{\mathrm{w}}} \to 0.
\end{equation}
In particular, if \(q_k+c_k < 0\), then \(H^1(X_{\mathrm{w}},S_{X_{\mathrm{w}}})_k \neq 0\).
\end{proposition}

\begin{proof}
Contract the arms to the cyclic quotient points of the partial resolution over the central \(\mathbb{P}^1\). At the \(i\)-th point use the quotient chart
\[
(u,v) \mapsto (\zeta u,\zeta^{a_i}v),\qquad
\zeta^{n_i} = 1,\qquad
z = u^{n_i},\qquad T = u^{-a_i}v.
\]
The central curve is \(v = 0\). A degree-\(k\) function \(g(z)T^k\) is regular if and only if \(n_i\operatorname{ord}(g) \geq ka_i\). The degree-\(k\) function sheaf on the central curve therefore has degree \(q_k\).

For the determinant line use
\[
\partial_z\wedge T\partial_T
 = \frac1{n_i}u^{1-n_i}\partial_u\wedge v\partial_v.
\]
Regularity requires \(n_i\operatorname{ord}(g)\geq ka_i+n_i-1\). Since \(\lceil(ka_i+n_i-1)/n_i\rceil=\lceil ka_i/n_i\rceil+1-\mathbf{1}_{\{ka_i\equiv1\pmod{n_i}\}}\), the resulting line bundle has degree
\[
2+bk-\sum_i\left\lceil\frac{ka_i+n_i-1}{n_i}\right\rceil=q_k-1+c_k.
\]
These conditions also give extension across the resolved arms. In a toric logarithmic bivector frame, a monomial has nonnegative exponents along the compact toric divisors and may have exponent \(-1\) along the noncompact outer divisor. The exponents satisfy \(r_{j+1} = b_{ij}r_j-r_{j-1}\). If an interior exponent becomes negative after a nonnegative predecessor, the following exponent is at most \(-2\) and the negative exponents subsequently decrease. Thus the regularity inequality along the outer divisor implies the regularity inequalities along all compact toric divisors. 

The partial resolution is affine over \(\mathbb{P}^1\): locally its projection is \(\operatorname{Spec}\mathbb{C}[u,v]^{\mu_{n_i}} \to \operatorname{Spec}\mathbb{C}[z]\), and globally it is the relative spectrum of the graded algebra with degree-\(k\) part \(\mathcal{O}_{\mathbb{P}^1}(q_k)\). The pushforward of the determinant line has degree-\(k\) part \(\mathcal{O}_{\mathbb{P}^1}(q_k-1+c_k)\). Lemma~\ref{star:local-direct-images} and Leray prove the formulas. A negative degree violates the regularity inequality along the central component, so no negative-degree part occurs. Finally, \(b-\sum_i a_i/n_i > 0\) makes the degrees of both line bundles tend to infinity. Only finitely many weight spaces of the cohomology groups are nonzero; the theorem on formal functions identifies their sum with the cohomology on a sufficiently small analytic resolution neighborhood.

The existence of the nowhere-zero Euler logarithmic section, including at the cyclic quotient resolutions, is proved by Wahl~~\cite[Proposition~3.1]{Wahl2015}. Its quotient is the determinant line, proving \eqref{star:euler-sequence}. Since \(H^2(X_{\mathrm{w}},\mathcal{O}_{X_{\mathrm{w}}}) = 0\), the resulting map \(H^1(S_{X_{\mathrm{w}}}) \to H^1(L_{X_{\mathrm{w}}})\) is onto in every degree. This proves the assertion.
\end{proof}

The sufficiency argument uses two results about rational singularities. In every application of these results below, all exceptional self-intersections are at most \(-2\), so the minimal good resolution is also the minimal resolution. For a rational singularity which is not a rational double point, Wahl~~\cite[Corollary~4.5(iii)]{Wahl2015} proves that \(h^1(L_X)\) is determined by the resolution graph. If its graph is noncyclic and star-shaped, then
\begin{equation}\label{star:rational-inequality}
h^1(S_X) \leq h^1(L_X).
\end{equation}
This inequality, with equality if and only if the singularity is quasi-homogeneous, is proved by Wahl~~\cite[Proposition~4.8]{Wahl2015}. These two results of Wahl suffice for the star-shaped graphs considered here. We treat rational double points separately, using the following lemma.

\begin{lemma}\label{star:all-two}
Let \((X,E)\) resolve a rational singularity, and suppose that \(E\) is a rational tree with every self-intersection equal to \(-2\). Then \(K_X\) is trivial on a sufficiently small resolution neighborhood and \(H^1(X,L_X) = 0\).
\end{lemma}

\begin{proof}
The sufficiently small resolution neighborhood retracts onto \(E\), as is seen from its tubular plumbing neighborhoods. Since \(H^1(X,\mathcal{O}_X) = H^2(X,\mathcal{O}_X) = 0\), the exponential sequence identifies \(\operatorname{Pic}(X)\) with \(H^2(X,\mathbb{Z})\). Because \(E\) is a rational tree, the latter group is identified with the component multidegrees. Adjunction gives \(K_X\cdot E_v = 0\) on every component, so \(K_X\) is trivial. Thus \(L_X \simeq \mathcal{O}_X(-E)\). The map \(H^0(X,\mathcal{O}_X) \to H^0(E,\mathcal{O}_E) = \mathbb{C}\) is onto by constants. The sequence \(0 \to \mathcal{O}_X(-E) \to \mathcal{O}_X \to \mathcal{O}_E \to 0\) and rationality give \(H^1(X,L_X) = 0\).
\end{proof}

Rational double points of types \(D\) and \(E\) are weighted homogeneous: by Brieskorn~~\cite[Theorems~2.10--2.11]{Brieskorn1968}, they are quotients of \(\mathbb{C}^2\) by finite linear groups, and scalar multiplication induces a positive grading on the invariant ring. Wahl~~\cite[Proposition~3.1]{Wahl2015} therefore gives the exact sequence
\[
0 \to \mathcal{O}_X \to S_X \to L_X \to 0.
\]
Rationality and Lemma~\ref{star:all-two} imply \(H^1(X,S_X)=0\).

\begin{lemma}\label{star:arithmetic}
Assume \(b=2\). The following numerical conditions are equivalent.
\begin{enumerate}
\item \(q_k+c_k\geq0\) for \(k=1,2,4\).
\item The arms belong to one of the three families in Theorem~\ref{star:classification}.
\item \(q_k\geq-1\) and \(q_k+c_k\geq0\) for every \(k\geq1\).
\end{enumerate}
Each of these three families is automatically negative definite.
\end{lemma}

\begin{proof}
Assume the first condition. Since \(q_1 = -1\), some arm satisfies \(a_i = 1\), so it consists of a single curve \([p]\). If another continued fraction is at least \(2\), then the first family occurs. Otherwise both other continued fractions are less than \(2\), and the integers \(\lceil2a_i/n_i\rceil\) are \(1,2,2\). Hence \(q_2 = -1\). The arm \([p]\) contributes nothing to \(c_2\), so one of the two other arms satisfies \(2a \equiv 1\pmod n\). Because \(n < 2a < 2n\), this is \(2a = n+1\). Its unique Hirzebruch--Jung expansion is
\[
\frac{n}{a} = \frac{2m-1}{m} = [2,m].
\]
If \(p \geq 3\), the second family occurs. If \(p = 2\), write \(n/a\) for the remaining continued fraction. If \(n/a \geq 4/3\), then the first alternative of the third family holds. If \(n/a < 4/3\), then the integers \(\lceil4a_i/n_i\rceil\) are \(2,3,4\), so \(q_4 = -1\). Neither \([2]\) nor \([2,m]\) contributes to \(c_4\), because \(4m \equiv 2\pmod{2m-1}\). Thus \(4a \equiv 1\pmod n\). Since \(3n < 4a < 4n\), one has \(4a = 3n+1\), which gives
\[
(n,a) = (4t+1,3t+1),\qquad
\frac{n}{a} = [2,2,2,t+1].
\]
This proves \((1)\Rightarrow(2)\) using only the three specified degrees, without a rationality assumption.

We prove the inequalities for every \(k \geq 1\) in each family. In the first family, order the arms so that the first is \([p]\) and \(n_2/a_2 \geq 2\). If \(p \geq 3\), then for every \(k \geq 2\),
\[
\left\lceil\frac{k}{p}\right\rceil+\left\lceil\frac{ka_2}{n_2}\right\rceil
 \leq \left\lceil\frac{k}{3}\right\rceil+\left\lceil\frac{k}{2}\right\rceil
 \leq k.
\]
The last inequality follows by writing \(k\) as \(2l\) or \(2l+1\). Since \(\left\lceil\frac{ka_3}{n_3}\right\rceil \leq k\), this gives \(q_k \geq 0\). At degree one the arm \([p]\) contributes to \(c_1\). If \(p = 2\), the sum \(\left\lceil\frac{ka_1}{n_1}\right\rceil+\left\lceil\frac{ka_2}{n_2}\right\rceil\) is at most \(2\lceil k/2\rceil\). Thus \(q_k \geq 0\) for even \(k\), while \(q_k \geq -1\) for odd \(k\), when the arm \([2]\) contributes to \(c_k\). This proves both inequalities in the first family.

In the second family put \(n = 2m-1\). Since \(p \geq 3\) and \(m/n \leq 2/3\),
\[
\left\lceil\frac{k}{p}\right\rceil+\left\lceil\frac{km}{n}\right\rceil
 \leq \left\lceil\frac{k}{3}\right\rceil+\left\lceil\frac{2k}{3}\right\rceil
 \leq k+1.
\]
Hence \(q_k \geq -1\). If \(q_k = -1\), equality holds in this bound and \(3\nmid k\). For \(k = 3l+1\), equality requires \(\lceil k/p\rceil = l+1\), hence \((p-3)l < 1\). Thus \(l = 0\) or \(p = 3\), and \(k \equiv 1\pmod p\); the arm \([p]\) contributes to \(c_k\). For \(k = 3l+2\), equality requires \(\lceil km/n\rceil = 2l+2\), hence \(l(m-2) < 1\). Thus \(l = 0\) or \(m = 2\), and \(km \equiv 1\pmod n\); the arm \([2,m]\) contributes to \(c_k\). This proves the inequalities in the second family.

In the third family put \(n = 2m-1\) and \(z = 1-a_3/n_3\). In its first alternative \(z \geq 1/4\). In the second alternative \(z = t/(4t+1) \geq 1/5\). For odd \(k\), substitution in the definition of \(q_k\) gives
\[
q_k = \lfloor kz\rfloor-\left\lceil\frac12+\frac{k}{2n}\right\rceil
 \geq \left\lfloor\frac{k}{5}\right\rfloor-\left\lceil\frac12+\frac{k}{6}\right\rceil
 \geq -1.
\]
Indeed, for odd \(k\),
\[
\left\lceil\frac{k}{6}-\frac12\right\rceil
 = \left\lfloor\frac{k+1}{6}\right\rfloor
 \leq \left\lfloor\frac{k}{5}\right\rfloor.
\]
For \(k \geq 5\) the last inequality follows from \((k+1)/6 \leq k/5\), and \(k = 1,3\) are immediate. The arm \([2]\) contributes to \(c_k\) in every odd degree.

For even \(k = 2l\), the formula becomes
\[
q_{2l} = \lfloor2lz\rfloor-\left\lceil\frac{l}{n}\right\rceil,
\]
and the arm \([2,m]\) contributes to \(c_{2l}\) if and only if \(l \equiv 1\pmod n\). If \(n \geq 5\), then
\[
q_{2l} \geq \left\lfloor\frac{2l}{5}\right\rfloor-\left\lceil\frac{l}{5}\right\rceil.
\]
Writing \(l = 5s+r\), \(0 \leq r < 5\), the expression on the right is \(s\) for \(r = 0,3,4\) and \(s-1\) for \(r = 1,2\). It is nonnegative except at \(l = 1,2\), where it is \(-1\). At \(l = 1\) the arm \([2,m]\) contributes to \(c_2\). At \(l = 2\), either \(z \geq 1/4\) gives \(q_4 \geq 0\), or the exceptional third arm contributes to \(c_4\).

It remains to take \(n = 3\). Define
\[
f(l) = \left\lfloor\frac{2l}{5}\right\rfloor-\left\lceil\frac{l}{3}\right\rceil.
\]
One has \(f(l+15)=f(l)+1\), and for \(1\leq l\leq15\) the values are
\[
f(l)=\begin{cases}
-1,&l\in\{1,2,4,7\},\\
1,&l=15,\\
0,&\text{otherwise}.
\end{cases}
\]
At \(l=1,4,7\), the arm \([2,2]\) contributes to \(c_{2l}\); the calculation for \(k=4\) treats \(l=2\). Thus the inequalities hold in every degree. This proves \((2)\Rightarrow(3)\); the implication \((3)\Rightarrow(1)\) is immediate.

For negative definiteness, in the first two families one has \(a_1/n_1+a_2/n_2\leq1\), while \(a_3/n_3<1\). In the third family, \(\sum_i a_i/n_i\leq1/2+2/3+4/5=59/30<2\).
\end{proof}

\begin{theorem}\label{star:classification}
If \(b \geq 3\), every negative definite graph \(\Gamma\) as above is rational and taut. If \(b = 1\), no such graph is taut. If \(b = 2\), the following conditions are equivalent.
\begin{enumerate}
\item The graph \(\Gamma\) is taut.
\item The inequalities \(q_k \geq -1\) and \(q_k+c_k \geq 0\) hold for every \(k \geq 1\).
\item Those inequalities hold for \(k = 1,2,4\).
\item After a permutation of the arms, one of the following holds:
\begin{enumerate}
\item the arms include \([p]\), \(p \geq 2\), and an arm with continued fraction \(n/a \geq 2\);
\item the arms include \([p]\), \(p \geq 3\), and \([2,m]\), \(m \geq 2\);
\item the arms are \([2]\), \([2,m]\), \(m \geq 2\), and an arm with continued fraction \(n/a \geq 4/3\), or with expansion \([2,2,2,t+1]\), \(t \geq 1\).
\end{enumerate}
\end{enumerate}
Every graph in the three displayed families is rational. The unspecified third arm in the first two families is arbitrary.
\end{theorem}

\begin{proof}
Suppose first that \(b = 1\). In degree one with respect to the normal coordinate of \(P\), one has \(q_1 = -2\), and \(c_1\) counts the arms consisting of a single curve. If \(c_1 = 0,1,2\), the restriction map from the central section space to the three jet quotients has cokernel of dimension \(3,2,1\), respectively, by Corollary~\ref{star:central-obstruction}. If \(c_1 = 3\), every arm consists of a single curve. In degree two, each integer \(\lceil2a_i/n_i\rceil\) equals one and \(c_2 = 0\). Thus \(q_2 = -1\), and the same corollary gives a nonzero cokernel of the restriction map. Hence \(\Gamma\) is not taut.

For \(b=2\), tautness and Corollary~\ref{star:central-obstruction} imply \(q_k+c_k\geq0\) for \(k=1,2,4\). Lemma~\ref{star:arithmetic} then gives the three families and both inequalities in every degree. The same lemma proves the equivalence of the second, third, and fourth conditions. Moreover, for \(b=1\), or for \(b=2\) outside these families, the obstruction occurs in degree one, two, or four. Thus Corollary~\ref{star:central-obstruction} gives a standard plumbing realization with \(H^1(rE,S_X|_{rE})\neq0\) for some \(r\in\{2,3,5\}\).

Conversely, suppose that \(b=2\) and the arms belong to a displayed family. Lemma~\ref{star:arithmetic} gives \(q_k\geq-1\) and \(q_k+c_k\geq0\) for every \(k\geq1\). Proposition~\ref{star:graded} gives \(H^1(\mathcal{O}_{X_{\mathrm{w}}}) = H^1(L_{X_{\mathrm{w}}}) = 0\). Rationality is a property of the weighted resolution graph by the numerical criterion for rational singularities of Artin~~\cite[Theorem~3, p.~132]{Artin1966}. For singularities other than rational double points, graph invariance gives \(H^1(L_X) = 0\), and \eqref{star:rational-inequality} gives \(H^1(S_X) = 0\), for every realization. For rational double points, the argument following Lemma~\ref{star:all-two} gives \(H^1(S_X) = 0\) directly. The cohomological criterion for tautness proves sufficiency.

If \(b\geq3\), then \(\lceil ka_i/n_i\rceil\leq k\) gives \(q_k\geq(b-3)k\geq0\) and \(q_k+c_k\geq0\) for every \(k\geq1\). The preceding cohomological argument again proves rationality and tautness. Lemma~\ref{star:arithmetic} also supplies the asserted negative definiteness of the three families.
\end{proof}

In the first family, an arm with continued fraction \(n/a \geq 2\) is either the single curve \([2]\), or its first component has self-intersection at most \(-3\). In the third family, the inequality \(n/a \geq 4/3\) means that either the arm has at most three components, each with self-intersection \(-2\), or at most two initial components have self-intersection \(-2\) before the first component with self-intersection at most \(-3\). These equivalences follow by iterating \(r \mapsto 2-1/r\) and impose no restriction on the remaining self-intersections.

\section{Local restrictions at two branching components}\label{sec:representations}

Let \(E\) be a rational tree with exactly two trivalent components \(P\) and \(Q\). Write \(P^2 = -b_P\), and let \(s_P\) be the number of its two outward arms containing at least two components; define \(b_Q,s_Q\) similarly. Let \(C = C_1+\cdots+C_m\) be the reduced chain strictly between \(P\) and \(Q\), in that order, and write \(C_j^2 = -c_j\). Put
\[
d = \sum_{j = 1}^m(c_j-2).
\]
If the connecting chain is empty, set \(d = 0\). Sequences describing arms are read away from the branching component. Every component other than \(P,Q\) has self-intersection at most \(-2\).

\subsection{Necessary local restrictions}

Cohomology on cycles contained in \(2E\) gives necessary conditions on the node self-intersections and outward arms. One-node subgraphs impose further restrictions on initial components of self-intersection \(-2\).

\begin{proposition}\label{rep:first-layer}
Put \(R_{\mathrm{end}} = \sum E_i\), where the sum is over the four end components, and set \(E' = E-R_{\mathrm{end}}\). Then
\[
H^0(E',S_X|_{E'}) = H^1(E',S_X|_{E'}) = 0.
\]
Consequently the cycle \(Z_2 = E+E' \leq 2E\) satisfies
\[
H^1(Z_2,S_X|_{Z_2}) \simeq H^1(E,S_X(-E')|_E).
\]
If \(b_P = 1\), or if \(b_P = 2,s_P = 2\), then every analytic realization has \(H^1(X,S_X) \neq 0\).
\end{proposition}

\begin{proof}
On a trivalent component \(S_X\) restricts to \(\mathcal{O}\oplus\mathcal{O}(-1)\); on a component of valency two it restricts to \(\mathcal{O}^{\oplus2}\). Removing the outward chains in the normalization sequence for \(E'\) leaves the two normal Euler lines and the chain joining them. In the ordered logarithmic bases the transport across a component of self-intersection \(-b\) is
\[
T_b = \begin{pmatrix}b&1\\-1&0\end{pmatrix}.
\]
The cone \(\alpha > 0,\beta \leq 0,\alpha+\beta \geq 0\) is preserved by \(T_b\) for \(b \geq 2\), since the image has coordinate sum \((b-2)\alpha+\alpha+\beta\). The left normal Euler line starts at \((1,0)\), whereas the right line, in the ordered logarithmic bases at the crossing, is \((0,1)\). They cannot coincide after transport. Thus \(H^0(E',S_X|_{E'}) = 0\). The normalization sequence gives \(\chi(E',S_X|_{E'}) = 0\), so \(H^1\) is also zero. The exact sequence
\[
0 \to S_X(-E')|_E \to S_X|_{E+E'} \to S_X|_{E'} \to 0
\]
gives the asserted isomorphism.

Set \(F = S_X(-E')\). At \(P\), the equality \(E'\cdot P = -b_P+1+s_P\) gives
\[
F|_P \simeq 
\mathcal{O}_{\mathbb{P}^1}(b_P-1-s_P)
\oplus\mathcal{O}_{\mathbb{P}^1}(b_P-2-s_P).
\]
Normalization gives a natural surjection \(H^1(E,F|_E) \to H^1(P,F|_P)\). For \(b_P = 2,s_P = 2\), its target contains \(H^1(\mathcal{O}_{\mathbb{P}^1}(-2))\). For \(b_P = 1,s_P \geq 1\), its dimension is \(2s_P-1\). If \(b_P = 1,s_P = 0\), project the normalization targets to the two edges joining \(P\) to its outward end components. The target has dimension four, while the only contributing section spaces have total dimension three: one from \(P\) and one from each end component. Thus its cokernel is nonzero. Restriction from \(X\) to \(Z_2\) is surjective on \(H^1\) by Lemma~\ref{fnd:restriction}, proving the last assertion.
\end{proof}

\begin{lemma}[Boundary values after matching the outward arms]\label{rep:boundary-values}
Assume that \(b_P,b_Q\geq2\), and that \(s_U\leq1\) whenever \(b_U=2\), for \(U=P,Q\). Put \(F=S_X(-E')\). Let \(E_P\) be the reduced union of \(P\) and its two outward arms, and let \(p_P\) be the point where \(P\) meets the connecting path towards \(Q\). Define
\[
W_P=\operatorname{im}\left(H^0(E_P,F|_{E_P})\to F|_{p_P}\right),
\]
and define \(W_Q\) similarly. In ordered logarithmic bases, these spaces are
\[
\begin{array}{c|c}
(b_P,s_P)&W_P\\ \hline
b_P \geq 4,\ \text{or }b_P = 3,s_P \leq 1&\mathbb{C}^2\\
b_P = 3,s_P = 2&\mathbb{C}(0,1)\\
b_P = 2,s_P = 0&\mathbb{C}(1,-\tau_P)\\
b_P = 2,s_P = 1&0
\end{array}
\qquad
\tau_P = 2-\frac1p-\frac1q \geq 1
\]
in the third row, where the two outward end components have self-intersections \(-p,-q\). If \(d=0\), then
\[
\begin{gathered}
H^1(E,F|_E)\simeq\mathbb{C}^2/(Q_mW_P+W_Q),\\
Q_m=T_2^mJ=\begin{pmatrix}m&m+1\\1-m&-m\end{pmatrix},\qquad
J=\begin{pmatrix}0&1\\1&0\end{pmatrix}.
\end{gathered}
\]
If \(d>0\), then \(H^1(E,F|_E)=0\). Consequently, if the graph is taut and \(b_P=2,s_P=1,d=0\), then
\begin{equation}
b_Q \geq 4,\qquad\text{or}\qquad b_Q = 3,\ s_Q \leq 1.
\label{rep:condition-b}
\end{equation}
\end{lemma}

\begin{proof}
Choose the connecting attachment on \(P\) to be \(x = 0\), and the outward attachments to be \(x = 1,-1\). A section of \(F|_P\) is
\[
A(x)x(x^2-1)\partial_x+B(x)y\partial_y,
\]
after choosing a local generator of \(\mathcal{O}_X(-E')\), with
\[
\deg A \leq b_P-2-s_P,\qquad
\deg B \leq b_P-s_P,\qquad
B_{b_P-s_P} = -b_P A_{b_P-2-s_P}.
\]
If the bound on \(A\) is negative, \(A = 0\) and the degree of \(B\) drops by one. Its value at the attachment to the connecting chain is \((-A(0),B(0))\). An outward end component of self-intersection \(-p\) imposes \(B(1) = -2A(1)/p\); at \(-1\) the analogous equation is \(B(-1) = -2A(-1)/q\). An arm of length at least two imposes no condition: at its last nonterminal component the twist has positive degree, so its two values can be prescribed independently, while all preceding components have nonnegative degree and evaluation from its terminal component has one-dimensional image.

These equations verify every row. If \(b_P = 2,s_P = 1\), \(A = 0\), \(B\) is constant, and the remaining end condition forces \(B = 0\). If \(b_P = 3,s_P = 2\), \(A = 0\) and \(B\) is an arbitrary constant. If \(b_P = 2,s_P = 0\), \(A = a\), \(B = -2ax^2+dx+c\), and the two equations give \(c = (2-1/p-1/q)a\), hence the stated line. For \(b_P = 3,s_P = 1\), choose \(A\) constant and \(B = -3Ax^2+dx+c\); its free coefficient \(d\) can be chosen to satisfy the condition at the outward attachment without changing \(A,c\). For \(b_P = 3,s_P = 0\), choose \(A\) constant and use the two free coefficients of a quadratic \(B\) other than its prescribed constant to satisfy the conditions at both outward attachments. For \(b_P = 4,s_P = 2\), the constant \(A\) and \(B = -4Ax^2+dx+c\) already give arbitrary values at the connecting attachment. For all larger remaining values, interpolation with the same constant \(A\) and quadratic \(B\) satisfies the degree bounds and the conditions on the highest coefficients, giving the full fiber.

All component first cohomology groups vanish under the hypotheses. When \(d=0\), normalization therefore reduces the cohomology to the displayed quotient. The transport along the \(m\) connecting components is \(T_2^m\); the final exchange of tangent and normal entries gives the factor \(J\). For the empty chain this reads \(Q_0=J\), so the same formula applies. If a connecting component has self-intersection at most \(-3\), the twist has degree at least one on that component. Evaluation at its two attachment points is then surjective onto the product of the fibers, and the remaining cokernel is zero. Finally, when \(b_P=2,s_P=1,d=0\), the table gives \(W_P=0\). By Proposition~\ref{rep:first-layer} and the restriction surjection, tautness requires \(W_Q=\mathbb{C}^2\), which is exactly \eqref{rep:condition-b}.
\end{proof}

\begin{lemma}\label{rep:prefix-witness}\label{rep:prefix-necessity}
Suppose every component of a negative-definite graph with two nodes has self-intersection at most \(-2\). Let \(P^2 = -2\), with outward arms \([p]\) and \(A\). Either of the following conditions gives an analytic realization and an effective cycle \(Y\) with \(H^1(Y,S_X|_Y) \neq 0\):
\begin{enumerate}
\item \(A\) has length at least three, and the first component of each of \(A\) and the connecting chain has self-intersection \(-2\);
\item \(p = 2\), \(A = [2,a]\), and the first three components of the connecting chain have self-intersection \(-2\).
\end{enumerate}
One may take \(Y \leq 3E\) in the first case and \(Y \leq 5E\) in the second. Consequently neither condition occurs in a taut graph.
\end{lemma}

\begin{proof}
In the first case retain the arm consisting of a single component, the first three components of \(A\), and a three-component path starting along the connecting chain. If that path reaches the other node, continue along one of its outward arms. Such a path exists even when the connecting chain has only one component. The full subgraph on these components is a star-shaped graph \(D\) with arms \([p],[2,u,v],[2,u',v']\). Its matrix is a principal submatrix of the negative-definite intersection matrix, and every component has self-intersection at most \(-2\). For \(n/a = [2,u,v]\),
\[
a = uv-1,\qquad n = 2uv-v-2,\qquad 2a-n = v \geq 2.
\]
In degree two in the normal coordinate of the central component the three vanishing orders in Proposition~\ref{star:central-bundle} are therefore \((1,2,2)\), and none of the three congruences \(ka_i \equiv 1\pmod{n_i}\) holds. Thus \(m = 4-5 = -1,c = 0\), and
\[
\mathcal{F}_2 \simeq \mathcal{O}_{\mathbb{P}^1}(-1)\oplus\mathcal{O}_{\mathbb{P}^1}(-2),\qquad h^1(\mathcal{F}_2) = 1.
\]
Corollary~\ref{star:central-obstruction} gives a nonzero obstruction on \(3D\) in the standard plumbing realization.

In the second case retain the two outward arms and a five-component path starting along the connecting chain, again continuing along one outward arm at the other node when necessary. The path exists because the connecting chain has at least three components. The star-shaped full subgraph formed by these components has arms \([2],[2,a],[2,2,2,u,v]\). For the last arm,
\[
\begin{gathered}
n'/a' = [2,2,2,u,v],\qquad a' = 3uv-2v-3,\qquad n' = 4uv-3v-4,\\
4a'-3n' = v \geq 2.
\end{gathered}
\]
In degree four in the normal coordinate of the central component the vanishing orders are \((2,3,4)\). The first arm does not satisfy the congruence condition; on the second, \(4a \equiv 2\pmod{2a-1}\); and the displayed inequality excludes the congruence condition on the last arm. Hence \(m = 8-9 = -1,c = 0\). The bundle \(\mathcal{F}_4\) is again \(\mathcal{O}(-1)\oplus\mathcal{O}(-2)\), and Corollary~\ref{star:central-obstruction} gives a nonzero obstruction on \(5D\).

In either case extend this standard realization to the full graph by Lemma~\ref{fnd:plumbing}. Lemma~\ref{fnd:forget-branches} gives a surjection from \(H^1(Y,S_X|_Y)\) to \(H^1(Y,\Theta_X(-\log D)|_Y)\), which is nonzero. This preserves the bounds \(3D \leq 3E\) and \(5D \leq 5E\). Restriction from \(X\) to that cycle and Theorem~\ref{fnd:tautness} prove the last assertion.
\end{proof}

\begin{proposition}\label{rep:local-restrictions}
For a taut graph with two nodes, every component has self-intersection at most \(-2\). For each ordered pair \((P,Q)\) with \(b_P=2\), one has \(s_P\leq1\). At each such node, if \(s_P=1\), then the outward arms and the connecting chain have one of the forms in Table~\ref{intro:two-node-table}, without imposing its last column. If \(s_P=1\) and \(d=0\), then \eqref{rep:condition-b} holds.
\end{proposition}

\begin{proof}
Every nonbranching component has self-intersection at most \(-2\), since the resolution is minimal and good. Proposition~\ref{rep:first-layer} excludes \(b_P=1\), and excludes \(s_P=2\) when \(b_P=2\). Apply the same argument at \(Q\).

Suppose \(b_P=2,s_P=1\), and write the outward arms as \([p]\) and \(A\). If the first component of \(A\) has self-intersection at most \(-3\), the first row with \(s_P=1\) in Table~\ref{intro:two-node-table} applies. Otherwise \(A\) begins with a component of self-intersection \(-2\). For an arm of length two, \(A=[2,a]\), the next row applies when \(p\geq3\). When \(p=2\), the second condition in Lemma~\ref{rep:prefix-witness} excludes three initial connecting components of self-intersection \(-2\). Thus their maximal initial sequence has length zero, one, or two, giving the next two rows of the table. If \(A\) has length at least three, the first condition in that lemma forces the connecting chain to be empty or to begin with a component of self-intersection at most \(-3\), giving the last row. These alternatives give all the required arm and chain forms without their last-column conditions.

When \(s_P=1,d=0\), Lemma~\ref{rep:boundary-values} gives \eqref{rep:condition-b}.
\end{proof}

\subsection{A cycle assembled from eight local cases}

We now assume that every component has self-intersection at most \(-2\) and that the restrictions stated in Proposition~\ref{rep:local-restrictions} hold. At a branching component \(P\), consider the maximal initial sequence of connecting components, read from \(P\), all of self-intersection \(-2\). Let \(K_1,\ldots,K_h\) be the connecting components included with \(P\) below. In cases (i)--(iii) none is included, so \(h = 0\). In cases (iv)--(vii) include this entire maximal sequence; in cases (v)--(vii) its length is \(0,1,2\), respectively, and the arm and chain forms in Table~\ref{intro:two-node-table} exclude greater lengths. In case (viii) the maximal sequence is empty by the arm and chain forms in Table~\ref{intro:two-node-table}, and \(h = 0\). In the table, coefficients are cycle multiplicities, and \(p,a,b_P\) are defined by the self-intersections \(-p,-a,-b_P\). Put \(T = P+K_1+\cdots+K_h\).

\begin{table}[ht]
\centering
\begin{tabular}{c|p{0.38\textwidth}|p{0.36\textwidth}|c}
Case&Local data&Cycle \(W\)&\(J\)\\ \hline
(i)&\(b_P = 2\), arms \([p]\), \(A = A_1+\cdots+A_l\), \(l \geq 2\), first component of self-intersection at most \(-3\)&\(2P+B_1+\sum_{j = 1}^lA_j\)&2\\
(ii)&\(b_P = 2\), two end components \(B_1,B_2\)&\(2P+B_1+B_2\)&2\\
(iii)&\(b_P \geq 3\), arbitrary two outward arms&twice \(P\) and every nonterminal outward component, once every outward end component&2\\
(iv)&\(b_P = 2\), arms \([p],[2,a]\), \(p \geq 3\)&\(3T+B_1+2A_1+A_2\)&3\\
(v)&\(b_P = 2\), arms \([2],[2,a]\), \(h = 0\)&\(3P+2B_1+2A_1+A_2\)&3\\
(vi)&same arms, \(h = 1\)&\(4P+2B_1+3A_1+2A_2+3K_1\)&3\\
(vii)&same arms, \(h = 2\)&\(6P+3B_1+4A_1+2A_2+5K_1+4K_2\)&4\\
(viii)&\(b_P = 2\), arms \([p]\), \(A_1,\ldots,A_l\), \(l \geq 3\), first component of self-intersection \(-2\)&\(4P+2B_1+\sum_{j = 1}^l\min(4,l+1-j)A_j\)&4
\end{tabular}
\caption{The eight effective cycles. The last column gives the coefficient of the last component of \(T\).}\label{rep:case-table}
\end{table}

The included sequences at the two ends do not overlap. If \(d > 0\), a component of self-intersection at most \(-3\) separates them. If \(d = 0\), the restriction in \eqref{rep:condition-b} prevents both ends from including connecting components.

These cases exhaust the possibilities allowed by the restrictions stated in Proposition~\ref{rep:local-restrictions}. Indeed a branching component of self-intersection \(-2\) has either two end components or one end component and one arm of length at least two. If the first component of that arm has self-intersection at most \(-3\), use (i). Otherwise an arm of length two gives (iv)--(vii); an arm of greater length forces \(h = 0\), giving (viii). A branching component of self-intersection at most \(-3\) gives (iii).

Let \(W_L,W_R\) be the resulting cycles at the two ends. If a nonempty connecting chain remains, denote its reduced cycle by \(K\), and set
\[
Z = W_L+2K+W_R.
\]
If all connecting components have been included in \(W_L\) or \(W_R\), put \(Z = W_L+W_R\). The reduced supports meet transversely at the indicated points, and the cycles meet along the corresponding scheme-theoretic intersections. When \(K = 0\), one of \(W_L,W_R\) is of case (iii), unless both are of case (ii). This follows from the restriction in \eqref{rep:condition-b}.

\begin{proposition}\label{rep:cycles}
Every coefficient of \(Z\) is at most six. If the condition in the last column of Table~\ref{intro:two-node-table} also holds, then \(Z\) is anti-nef. The construction, including the restrictions of \(S_X\), is defined for every analytic realization and uses no choice of a weighted homogeneous structure on the full graph.
\end{proposition}

\begin{proof}
The coefficient bound follows from Table~\ref{rep:case-table}. We first compute the intersections for an endpoint piece whose adjacent exterior component has coefficient two; direct intersections with an endpoint piece of larger coefficient are treated below. All outward intersections are nonpositive by direct substitution. For example case (vii) gives intersections zero at \(P,B_1,A_1,K_1\), \(4-2a\) at \(A_2\), and \(-1\) at \(K_2\), when the adjacent remaining component has coefficient two. Case (viii) gives \(n_1-4 \leq 0\) at \(P\), \(4-2p \leq 0\) at \(B_1\), and nonpositive second differences along \(n_j = \min(4,l+1-j)\). In case (iv), the node intersection is \(-1\) when \(h = 0\), and zero when \(h > 0\); its last connecting component has intersection \(-1\). Cases (v), (vi) give the same conclusion by their displayed coefficients. In case (i), the intersection with the first component of the arm of length at least two is \(3-b_{A_1} \leq 0\); in cases (ii), (iii), the node intersections are respectively zero and at most \(-2b_P+4+s_P \leq 0\).

If an endpoint contributes \(J > 2\), its adjacent component of \(K\), when present, has self-intersection at most \(-3\), by maximality of the included sequence and the arm and chain forms in Table~\ref{intro:two-node-table}. Consequently every component of self-intersection \(-2\) remaining in \(K\) has neighboring cycle coefficients at most two. A component of self-intersection \(-c\), where \(c \geq 3\), at an end of a connecting chain of length at least two has intersection at most \(-2c+4+2 \leq 0\). For a single remaining component the condition is \(J_L+J_R \leq 2c\). It is automatic for \(c \geq 4\). For \(c = 3\), a piece of case (ii) or (iii) contributes \(J = 2\); if both branching components have self-intersection \(-2\), the condition in the last column of Table~\ref{intro:two-node-table} excludes cases (vii),(viii), leaving \(J_L,J_R \leq 3\).

When \(K = 0\) and one branching component is of case (iii), the intersection with that component is
\[
-2b_Q+2+s_Q+J.
\]
For \(J \leq 2\) this is nonpositive. For \(J = 3\), the restriction in \eqref{rep:condition-b} gives \(s_Q \leq 1\) when \(b_Q = 3\); for \(J = 4\), the condition in the last column of Table~\ref{intro:two-node-table} gives \(s_Q = 0\) when \(b_Q = 3\). For \(b_Q \geq 4\), the displayed intersection is nonpositive. When \(K = 0\) and both endpoint cycles are of case (ii), all intersections are also nonpositive. This proves anti-nefness.
\end{proof}

\subsection{Comparison with one-node configurations}

We use taut one-node configurations to compute restriction maps near the endpoint cycles. The comparison also gives the cohomology vanishing and lifting sequences in Lemma~\ref{rep:piece-vanishing}.

\begin{lemma}\label{rep:neighborhood-comparison}
Let \(W\) be an exceptional cycle, let \(C\) be a component not contained in its support, and put \(D = W_{\mathrm{red}}+C\). Suppose that \(D\) is connected, its graph is a negative-definite taut minimal good resolution graph, and every component of \(E-D\) is disjoint from \(W_{\mathrm{red}}\). A neighborhood of \(D\) is analytically isomorphic to a neighborhood of the exceptional divisor in any fixed realization of its graph. Transport \(W,C\) by this isomorphism. For every \(m\geq1\), the induced identifications of functions and logarithmic sections on \(W\) and \(W\cap mC\) commute with their restriction maps, multiplication of functions, and multiplication of sections by functions. The logarithmic bundle on the fixed realization is taken with respect to its full exceptional divisor. In the standard star-shaped realizations used below, the isomorphism may be chosen to preserve prescribed component labels.
\end{lemma}

\begin{proof}
Contract \(D\) on a sufficiently small neighborhood by the contraction theorem of Grauert~~\cite{Grauert1962}. The graph hypothesis makes this contraction the minimal good resolution of the resulting normal surface singularity. Tautness identifies this singularity with the one associated with the fixed realization. By the uniqueness of the minimal good resolution proved by Laufer~~\cite[Chapter~V]{Laufer1971}, this isomorphism lifts to an analytic isomorphism of neighborhoods of the exceptional divisors. Indeed, composing a resolution map with the singularity isomorphism gives another minimal good resolution of the same singularity, so uniqueness identifies the resolutions. The lift is unique because its restriction outside the exceptional divisor is prescribed. Label the components of the target by their inverse images under this lift. In every application below, the fixed realization is a standard plumbing realization of a trivalent star with all arm self-intersections at most \(-2\). If its components have already been labelled, the induced map differs from the prescribed labelling by a weighted graph automorphism \(\sigma\). Compose the lift with the biholomorphism inducing \(\sigma^{-1}\) from Lemma~\ref{star:plumbing-automorphisms}. In these applications the resulting lift preserves the prescribed labels, and hence the prescribed cycle \(W\) and component \(C\).

The omitted components are disjoint from the support of \(W\). Hence \(\Theta_X(-\log E) = \Theta_X(-\log D)\) on an open neighborhood of that support. The lift identifies the ambient logarithmic bundles and preserves the Cartier ideals defining \(W\) and \(mC\). It therefore identifies their scheme-theoretic intersection and both restriction maps. Since the identification comes from an analytic isomorphism, it also preserves multiplication of functions and its action on sections.
\end{proof}

\begin{lemma}\label{rep:piece-vanishing}
For every cycle \(W\) in Table~\ref{rep:case-table},
\[
H^1(W,S_X|_W) = H^1(W,\mathcal{O}_W) = 0.
\]
If \(K \neq 0\), then \(S_X|_K \simeq \mathcal{O}_K^{\oplus2}\), the line bundle \(I = \mathcal{O}_K(-K)\) has nonnegative component degrees, and
\[
H^1(2K,S_X|_{2K}) = H^1(2K,\mathcal{O}_{2K}) = 0.
\]
There are exact lifting sequences
\[
0 \to H^0(K,I)^{\oplus2} \to H^0(2K,S_X|_{2K}) \to H^0(K,S_X|_K) \to 0,
\]
\[
0 \to H^0(K,I) \to H^0(2K,\mathcal{O}_{2K}) \to \mathbb{C} \to 0.
\]
These conclusions do not require \(Z\) to be anti-nef.
\end{lemma}

\begin{proof}
Let \(C\) be the unique component outside \(W_{\mathrm{red}}\) that meets \(W_{\mathrm{red}}\) along the path between the nodes. Complete \(W_{\mathrm{red}}\) by \(C\), with coefficient one, and omit the other branches incident to \(C\). This gives a one-node graph. The forgotten branches meet the added component away from \(W\), so the logarithmic pair on the completion agrees with the full pair in a formal neighborhood of \(W\).

The completed graphs are taut by Theorem~\ref{star:classification}. In cases (i),(ii), one arm consists of a single component and another has continued fraction at least \(2\). In case (iii), the branching component has self-intersection at most \(-3\). In case (iv), the arms \([p]\), \([2,a]\), \(p \geq 3\), permit an arbitrary third arm. In cases (v)--(vii) the third arm is \([c]\), \([2,c]\), or \([2,2,c]\), with \(c \geq 3\), whose continued fractions are at least \(4/3\). In case (viii) the third arm is the single component \([c]\), \(c \geq 3\). In cases (v)--(viii), if \(K=0\), the observation following the definition of \(Z\) shows that \(C\) is the opposite node of case (iii), so its self-intersection is at most \(-3\). Each listed completion is rational by the same theorem, so restriction also gives \(H^1(W,\mathcal{O}_W) = 0\). Thus each completion has vanishing logarithmic \(H^1\), and restriction to \(W\) is surjective on \(H^1\) by Lemma~\ref{fnd:restriction}. This proves the first assertion. The completed divisor is a subconfiguration of the given resolution, and its intersection matrix is negative definite. Lemma~\ref{rep:neighborhood-comparison} identifies its function and logarithmic restriction maps near \(W\) with those in the standard realization of the graph with one node. This analytic identification justifies the calculations below using Theorem~\ref{star:classification}.

For the second assertion, each component of \(K\) has valency two in the graph of the full exceptional divisor. Propagating trivializations through the rational tree gives \(S_X|_K \simeq \mathcal{O}_K^{\oplus2}\). Set \(I = \mathcal{O}_K(-K)\). Its degree on a component of self-intersection \(-c_i\) is \(c_i-\deg_K(K_i) \geq 0\), so \(H^1(K,I) = 0\). The exact sequence
\[
0 \to S_X(-K)|_K \to S_X|_{2K} \to S_X|_K \to 0
\]
has both outer \(H^1\) groups zero. The analogous sequence for the structure sheaf \(0 \to I \to \mathcal{O}_{2K} \to \mathcal{O}_K \to 0\) gives \(H^1(\mathcal{O}_{2K}) = 0\). Since \(H^1(K,I) = 0\), taking global sections in both sequences gives the asserted lifting sequences.
\end{proof}

\subsection{Endpoint restriction images}

Suppose first that \(K \neq 0\); the case \(K=0\) is treated in Remark~\ref{rep:direct-endpoint}. Write \(D_i = W_i\cap2K\), \(i = L,R\). At a crossing choose \(t_i\) tangent to \(K\), \(n_i\) normal to \(K\), so that the last component of \(W_i\) is \(t_i = 0\). Then
\[
Q_i = H^0(D_i,\mathcal{O}_{D_i})
 = \mathbb{C}[t_i,n_i]/(t_i^{J_i},n_i^2),
\qquad
T_i = \mathbb{C}[t_i]/(t_i^{J_i}).
\]
Put \(P_i = H^0(D_i,S_X|_{D_i})\), a free rank-two \(Q_i\)-module. Define the image algebras
\[
A_i = \operatorname{im}\left(H^0(W_i,\mathcal{O}_{W_i}) \to Q_i\right),\qquad
A_K = \operatorname{im}\left(H^0(2K,\mathcal{O}_{2K}) \to Q_L\oplus Q_R\right),
\]
and the images of the section restriction maps
\[
V_i = \operatorname{im}\left(H^0(W_i,S_X|_{W_i}) \to P_i\right),\qquad
V_K = \operatorname{im}\left(H^0(2K,S_X|_{2K}) \to P_L\oplus P_R\right).
\]
They are modules over \(A_i,A_K\), respectively: lift both an image function and an image section, multiply the lifts on the same cycle, and restrict. This proof gives no action by arbitrary elements of \(Q_i\). In general \(V_i\) is not a \(Q_i\)-submodule.

The maps defining \(A_K,V_K\) restrict simultaneously to \(D_L\) and \(D_R\). For \(I = \mathcal{O}_K(-K)\), the square-zero sequence gives
\[
A_K = \mathbb{C}(1,1)\oplus
\operatorname{im}\left(H^0(K,I) \to n_LT_L\oplus n_RT_R\right).
\]
The second summand has square zero; its two projections need not be independent.

Let \(\overline{V}_i\) be the reduction of \(V_i\) modulo \(n_i\), and put \(N_i = V_i\cap n_iT_i^{\oplus2}\). Use the ordered logarithmic coefficient basis \((t_i\partial_{t_i},n_i\partial_{n_i})\). Write
\[
\tau = 2-\frac1p-\frac1q
\]
in case (ii), and \(w_h = (h+1,-h)\) in case (iv).

\begin{proposition}\label{rep:endpoint-images}
After compatible choices of local coordinates and a local frame of \(S_X\),
\[
\overline{V}_i = \mathbb{C}w_i+t_iT_i^{\oplus2}.
\]
Function restriction maps onto \(T_i\), and
\begin{equation}\label{cls:scalar-orders}
A_i\cap n_iT_i=n_it_i^{\rho_i}T_i.
\end{equation}
The constant line vector \(w_i\), the kernel \(N_i\) with its factor \(n_i\) suppressed, and the scalar order \(\rho_i\) are as follows:
\begin{equation}
\begin{array}{c|c|c|c|c}
\text{case}&J_i&w_i&n_i^{-1}N_i&\rho_i\\ \hline
\text{(i)}&2&(1,0)&0&2\\
\text{(ii)}&2&(1,0)&t_i\mathbb{C}(\tau,-1)&2\\
\text{(iii)},\ b_P = 3,s_P = 2&2&(1,0)&t_i\mathbb{C}(1,0)&1\\
\text{(iii)},\ \text{otherwise}&2&(1,0)&t_iT_i^{\oplus2}&1\\
\text{(iv)}&3&w_h&t_i^2T_i^{\oplus2}&2\\
\text{(v)}&3&(1,0)&t_i^2T_i^{\oplus2}&2\\
\text{(vi)}&3&(2,-1)&t_i^2T_i^{\oplus2}&2\\
\text{(vii)}&4&(3,-2)&t_i^2\mathbb{C}(3,-2)+t_i^3T_i^{\oplus2}&2\\
\text{(viii)}&4&(1,0)&t_i^2\mathbb{C}(1,0)+t_i^3T_i^{\oplus2}&2
\end{array}
\label{rep:nil-table}
\end{equation}
\end{proposition}

\begin{proof}
Cases (i)--(iii) follow from Proposition~\ref{app:local:doubled-node}; cases (iv)--(vii) follow from Proposition~\ref{rep:absorbed-string-images}; and case (viii) is Lemma~\ref{app:local:case-eight}. These calculations use the full intersection schemes throughout.

A coordinate change preserving the two ordered branches has the form \(t'=tU(t,n)\), \(n'=nV(t,n)\), with \(U,V\) units. In the ordered logarithmic frames its Jacobian is
\[
J_{\log}=\begin{pmatrix}
1+tU_t/U&nU_n/U\\
tV_t/V&1+nV_n/V
\end{pmatrix}.
\]
On \(nT^{\oplus2}\), this matrix is reduced modulo \(n\), and its reduction modulo \(t\) is the identity. After also changing the generators \(t,n\), it therefore acts on each quotient \(nt^rT^{\oplus2}/nt^{r+1}T^{\oplus2}\) by a common nonzero scalar. It preserves the filtration and the leading lines in the table. The reduced image contains \(tT^{\oplus2}\), so it is determined by its constant line. The scalar ideal in \eqref{cls:scalar-orders} is preserved as well: surjectivity onto \(T_i\) makes it a \(T_i\)-submodule, and branch-preserving changes preserve its colength in \(n_iT_i\). The analytic identification in Lemma~\ref{rep:neighborhood-comparison} transports both incident function and section maps by this same coordinate and frame change; the maps are not normalized independently.
\end{proof}

\begin{remark}\label{rep:direct-endpoint}
Suppose that \(K=0\). Let \(C_i\) be the component of \(W_i\) meeting \(W_j\), where \(\{i,j\}=\{L,R\}\). If the coefficient of \(C_j\) in \(W_j\) is two, then
\[
D=W_i\cap W_j=W_i\cap2C_j,\qquad
H^0(D,\mathcal{O}_D)=\mathbb{C}[t,n]/(t^{J_i},n^2),
\]
where \(C_i=\{t=0\}\) and \(C_j=\{n=0\}\). Proposition~\ref{rep:endpoint-images} computes the function and section restriction images from \(W_i\). Indeed, Lemma~\ref{rep:neighborhood-comparison} applies with \(C=C_j\), and the local calculations use only this intersection and the components of \(W_i\).

If both coefficients are two, apply the proposition to both pieces, interchanging \(t\) and \(n\). Otherwise, the piece with coefficient two is of case (iii). Its restriction to the full intersection is computed by Proposition~\ref{app:local:doubled-node} with \(h=J_i\), and its logarithmic section image is described in Lemma~\ref{int:U-jets}. These alternatives exhaust the possibilities by the observation following the definition of \(Z\).
\end{remark}

\subsection{Local calculations}\label{app:local:calculations}

We prove the endpoint formulas in Proposition~\ref{rep:endpoint-images}. Lemma~\ref{rep:neighborhood-comparison} permits calculations in standard plumbing coordinates and transports both incident restriction maps by the same analytic isomorphism. All section spaces below are restrictions of the ambient logarithmic tangent bundle.

\subsubsection{Monomial restrictions}

The recurrence used along an arm also computes restrictions on cycles with unequal component multiplicities.

\begin{lemma}\label{rep:component-recursion}\label{star:nonuniform-arm}\label{star:arm-threshold}
Let successive components have self-intersections \(-b_i\), where \(b_i\geq2\), and positive cycle multiplicities \(\nu_i\). A monomial with initial exponents \(r_0=k,r_1=j\) has subsequent exponents and logarithmic coefficient transitions
\[
r_{i+1}=b_i r_i-r_{i-1},\qquad
T_{b_i}=\begin{pmatrix}b_i&1\\-1&0\end{pmatrix}.
\]
While the exponents on compact components satisfy \(0\leq r_i<\nu_i\), compatibility transports the coefficient vector by these invertible matrices. At the first compact component where this fails, the coefficient is arbitrary if \(r_i\geq\nu_i\), and zero if \(r_i<0\). If no compact cutoff occurs along a complete arm with continued fraction \(n/a\), regularity at the outer chart is imposed separately. Its terminal exponent is \(nj-ak\), and the permitted initial coefficient space is \(U_{k,j}\) from Lemma~\ref{star:vector-arm}: terminal exponent \(-1\) permits the line \(nA+aB=0\), terminal exponent at most \(-2\) forces vanishing, and nonnegative terminal exponent permits every coefficient vector. For functions the terminal exponent must be nonnegative. The conditions for distinct monomials are independent.
\end{lemma}

\begin{proof}
At a crossing, the intersection of components of multiplicities \(\nu_i,\nu_{i+1}\) has ideal \((u^{\nu_i},v^{\nu_{i+1}})\). A monomial survives precisely when its exponents are nonnegative and below these bounds. Reaching a compact multiplicity makes its restriction zero, so the remaining section extends by zero across the preceding intersection; a negative compact exponent violates regularity and forces its coefficient to vanish. The coordinate transition gives the recurrence and the invertible matrix above. If no compact cutoff occurs, Lemma~\ref{star:vector-arm} gives the terminal condition, including the exceptional line at exponent \(-1\). These steps also construct every permitted section. Since transitions and cycle ideals are monomial, the compatibility equations for distinct monomials are independent. The same proof applies to functions, omitting the coefficient matrices and using nonnegative terminal exponent.
\end{proof}

For a sequence of components of self-intersection \(-2\), the recurrence has the following explicit form.

\begin{lemma}\label{rep:initial-string-characters}
Let \(P,K_1,\ldots,K_h\) be successive components, where \(K_i^2=-2\), with cycle multiplicities \(\mu_0,\ldots,\mu_h\), and let the adjoining exterior multiplicity be \(a\). For an outgoing monomial \(v^qu^\ell\), where \(0\leq q<\mu_h\) and \(0\leq\ell<a\), its preceding exponents are
\[
r_i=(h+1-i)q-(h-i)\ell,\qquad 0\leq i\leq h+1.
\]
Apply Lemma~\ref{rep:component-recursion} backwards. If no exponent becomes negative or reaches its multiplicity before \(P\), the outgoing coefficient is \(T_2^h a_{r_0,r_1}\), where \(a_{k,j}\) is the central coefficient of normal degree \(k\) and tangent order \(j\). The central coefficients must also satisfy regularity at every intermediate chart.
\end{lemma}

\begin{proof}
The displayed exponents solve \(r_{i+1}=2r_i-r_{i-1}\) with \((r_h,r_{h+1})=(q,\ell)\). The preceding lemma gives the coefficient rule and the regularity conditions.
\end{proof}

\subsubsection{Coefficients on the node}

At a trivalent component \(P^2=-b\), put its outward attachments at \(z=0,1\) and its external attachment at infinity. A section of central normal degree \(k\) is
\begin{equation}\label{app:local:node-formula}
x^k\left(f_k(z)x\partial_x+g_k(z)z(z-1)\partial_z\right),\qquad
\deg f_k\leq bk,\quad\deg g_k\leq bk-1.
\end{equation}
A negative degree bound means zero. With \(u=z^{-1}\), \(v=z^b x\), its coefficient in the ordered frame \((e_n,e_t)=(v\partial_v,u\partial_u)\) is
\begin{equation}\label{app:local:infinity}
v^k u^{bk}\left(f_k(1/u)+b(1/u-1)g_k(1/u),-(1/u-1)g_k(1/u)\right)
=v^k\sum_j a_{k,j}u^j.
\end{equation}
Consequently \(a_{k,0}=\cdots=a_{k,d-1}=0\) if and only if \(\deg f_k\leq bk-d\) and \(\deg g_k\leq bk-d-1\): first use the tangent coefficient and then the normal coefficient.

\begin{proposition}\label{app:local:doubled-node}
Let \(W=2P+\sum_j\nu_j C_j\), where the \(C_j\) are the components of the two complete outward arms and \(\nu_j\geq1\), and let \(s\) count the arms of length at least two. The bundle of section coefficients of normal degree one and its rank on jets of length \(h\) at the external attachment are
\[
\mathcal{F}_1\simeq\mathcal{O}_{\mathbb{P}^1}(b-2)\oplus\mathcal{O}_{\mathbb{P}^1}(b-s-1),\qquad
\min(b-1,h)+\min(b-s,h).
\]
The function image on \(W\cap hC\), with ring \(\mathcal{R}=\mathbb{C}[u,v]/(u^h,v^2)\), is \(A=\mathbb{C}\oplus v(1-u)P_{b-2}\), where \(P_d\) is the image of polynomials of degree at most \(d\) in \(\mathbb{C}[u]/(u^h)\). For \(h=2\) and \((b,s)\neq(2,2)\), the reduced section image modulo \(u\) is \(\mathbb{C}e_n+v\mathbb{C}^2\), and the kernel is
\[
\begin{array}{c|c}
(b,s)&V\cap u\mathcal{R}^{\oplus2}\\\hline
(2,1)&0\\
(2,0)&uv\mathbb{C}(2-1/p-1/q,-1)\\
(3,2)&uv\mathbb{C}(1,0)\\
b\geq3,\ (b,s)\neq(3,2)&uv\mathbb{C}^2.
\end{array}
\]
Here \(-p,-q\) are the outward self-intersections when \(s=0\).
\end{proposition}

\begin{proof}
In central normal degree one, tangent exponent zero gives the initial exponents \((1,0)\). An arm of length at least two then gives exponent \(-1\), so both coefficients must vanish at its attachment. A single component of self-intersection \(-p\) instead imposes the line condition \(pA+B=0\). Every tangent exponent at least one extends, since its subsequent exponents are nonnegative and nondecreasing. Thus the common vanishing order at the two attachments is \(s\), and there are \(2-s\) additional line conditions. The proof of Proposition~\ref{star:central-bundle}, now at two attachments, gives the displayed splitting. Evaluation of \(\mathcal{O}(d)\) on jets of length \(h\) has rank \(\min(\max(d+1,0),h)\), proving the rank formula.

Functions in normal degree one are \(xz(z-1)H(z)\), with \(\deg H\leq b-2\); their image at infinity is \(v(1-u)u^{b-2}H(1/u)\), proving the formula for \(A\). Degree zero gives constants and the normal Euler vector.

For \(b=2,s=1\), put the short arm \([p]\) at zero. The sections are
\[
f_1=(z-1)(\alpha z+\beta),\qquad g_1=p\beta(z-1),\qquad
a_{1,0}=(\alpha+2p\beta,-p\beta).
\]
This coefficient map is invertible, giving zero kernel. For \(b=2,s=0\), write \(f_1=\alpha z^2+\beta z+\gamma\), \(g_1=\delta z+\eta\). The arm equations are \(\eta=p\gamma\), \(\delta+\eta=-q(\alpha+\beta+\gamma)\). The constant coefficient \((\alpha+2\delta,-\delta)\) is arbitrary; setting it to zero leaves linear coefficient \(\eta(2-1/p-1/q,-1)\). For \(b=3,s=2\), write \(f_1=z(z-1)(\alpha z+\beta)\), \(g_1=\gamma z(z-1)\). Vanishing of \((\alpha+3\gamma,-\gamma)\) leaves \(\beta e_n\) in order one. In all other cases with \(b\geq3\), both summands have degree at least one, so evaluation on length-two jets is surjective.
\end{proof}

\subsubsection{Cases (iv)--(vii)}

For these cases \(b=2\). Write \(B_1^2=-p\), \(A_1^2=-2\), \(A_2^2=-a\), and use the coefficients \(a_{k,j}\) in \eqref{app:local:infinity}.

\begin{proposition}\label{rep:absorbed-string-images}
In cases (iv)--(vii), the reduced section images and their nilpotent kernels are the corresponding rows of \eqref{rep:nil-table}. Function restriction is onto the reduced intersection algebra, and its nilpotent kernel has order two, as in \eqref{cls:scalar-orders}.
\end{proposition}

\begin{proof}
We separate the central coefficient calculation, transport along the included string, and function restriction.

\emph{Central coefficients.} In normal degree one use the \((b,s)=(2,1)\) polynomials in Proposition~\ref{app:local:doubled-node}. Their constant coefficient \(a_{1,0}\) is arbitrary, and \(a_{1,0}=0\) forces the entire degree-one section, including \(a_{1,1}\), to vanish. In degree two, Lemma~\ref{star:nonuniform-arm} gives
\[
\begin{gathered}
f_2=z(z-1)(\alpha z^2+\beta z+\gamma),\qquad
g_2=z(z-1)(\delta z+\epsilon),\\
(2a-1)(\alpha+\beta+\gamma)+a(\delta+\epsilon)=0.
\end{gathered}
\]
The short arm forces the common zero at zero. The other arm forces the common zero at one and the displayed line condition. The pair \((a_{2,0},a_{2,1})\) has four coordinates
\[
(\alpha+2\delta,-\delta,\ \beta-\alpha+2\epsilon-4\delta,2\delta-\epsilon).
\]
They are arbitrary: they determine \(\alpha,\beta,\delta,\epsilon\), and the arm equation then determines \(\gamma\).

For cases (vi),(vii), the degree-three arm conditions are
\[
z(z-1)^2\mid f_3,g_3,\qquad 2f_3'(0)-g_3'(0)=0.
\]
The short arm gives the zero and line condition at zero. On \([2,a]\), tangent exponents zero and one are forbidden, whereas exponent two gives \((3,2,1)\) with terminal exponent \(a-2\geq0\). Once the connecting-string conditions \(a_{3,0}=a_{3,1}=0\) are imposed, the solution is
\[
f_3=z(z-1)^2(cz+d),\qquad
g_3=2dz(z-1)^2,\qquad
a_{3,2}=(c+4d,-2d).
\]
Thus \(a_{3,2}\) is arbitrary. In case (vii), the degree-four arm rule is \(z^2(z-1)^3\mid f_4,g_4\). At the short arm tangent exponents zero and one are forbidden. At the long arm exponent two gives terminal exponent \(-2\), whereas exponent three reaches the multiplicity of \(A_2\). The connecting-string conditions \(a_{4,0}=a_{4,1}=a_{4,2}=0\) therefore leave precisely
\[
f_4=c z^2(z-1)^3,\qquad g_4=0,\qquad a_{4,3}=c(1,0).
\]

\emph{Transport to the external intersection.} For \(h=0\), degrees one and two already prove rows (iv),(v): reduction in degree one is onto and has zero kernel, and the two degree-two coefficients are independent. For the other cases, Lemma~\ref{rep:initial-string-characters} gives the following regularity conditions:
\[
\begin{array}{c|l}
\text{case}&\text{coefficients required to vanish}\\ \hline
\text{(iv), }h=1&a_{1,0},\ a_{2,0}\\
\text{(iv), }h\geq2&a_{1,0},\ a_{2,0},\ a_{2,1}\\
\text{(vi)}&a_{1,0},\ a_{2,0},\ a_{3,0},\ a_{3,1}\\
\text{(vii)}&a_{1,0},\ a_{2,0},a_{2,1},\ a_{3,0},a_{3,1},\
a_{4,0},a_{4,1},a_{4,2},\ a_{5,0},\ldots,a_{5,3}.
\end{array}
\]
These are all the connecting-string conditions: the excluded exponent sequences decrease to a negative value, and every other sequence remains nonnegative until it reaches a multiplicity or the outgoing crossing. In case (vi), the degree-two calculation above still applies, since tangent exponent two reaches the multiplicity of \(A_2\). In case (vii), degrees two and five may be set to zero, since neither contributes a required outgoing term.

Write \([M]\) for the outgoing coefficient vector of a monomial \(M\), and let \(\ast\) denote an independently arbitrary vector. Backward propagation gives
\[
\begin{array}{c|c|c|c}
\text{case}&[v]&[uv]&[uv^2]\\ \hline
\text{(iv), }h=1&T_2a_{2,1}&T_2a_{1,1}=0&\ast\\
\text{(iv), }h\geq2&\ast&T_2^ha_{1,1}=0&\ast\\
\text{(vi)}&T_2a_{2,1}&T_2a_{1,1}=0&T_2a_{3,2}\\
\text{(vii)}&T_2^2a_{3,2}&T_2^2a_{1,1}=0&T_2^2a_{4,3}.
\end{array}
\]
In addition, the coefficient of \(v^2\) is free in every row, and those of \(v^3,uv^3\) are free in case (vii). For example, in case (vii) the three displayed transported coefficients have central pairs \((k,j)=(3,2),(1,1),(4,3)\). Every coefficient marked or described as free comes from a monomial that vanishes on a preceding intersection after reaching a cycle multiplicity. Such a monomial extends independently by zero across that intersection, by Lemma~\ref{rep:component-recursion}. The central coefficient maps above show that the other asserted parameters are independent as well. The constant line is transported to \(\mathbb{C}T_2^h(1,0)\), where
\[
T_2^h(1,0)=(h+1,-h).
\]
The table therefore proves rows (iv)--(vii) of \eqref{rep:nil-table}; in particular \(a_{4,3}=c(1,0)\) gives the line \(\mathbb{C}(3,-2)\) in case (vii).

\emph{Function restriction.} Put \(B=\mathbb{C}[u]/(u^2)\). For cases (iv),(v), the central function spaces in degrees one and two are \(\mathbb{C}z(z-1)\) and \(z(z-1)^2\mathbb{C}[z]_{\leq1}\). The latter has arbitrary first two infinity coefficients. Applying the same exponent recurrence gives
\[
\begin{array}{c|c}
\text{case}&\text{function image}\\ \hline
\text{(iv), }h=0,\ \text{or (v)}
&\mathbb{C}+\mathbb{C}v(1-u)+v^2B\\
\text{(iv), }h\geq1,\ \text{or (vi)}
&\mathbb{C}+\mathbb{C}v+v^2B\\
\text{(vii)}
&\mathbb{C}+\mathbb{C}v+v^2B+v^3B.
\end{array}
\]
In case (vi), the degree-three polynomial \(z^2(z-1)^2\) supplies \(uv^2\). In case (vii), that polynomial supplies \(v\), and the degree-four polynomial \(z^2(z-1)^3\) supplies \(uv^2\); the remaining coefficients are free at preceding intersections. Each displayed algebra maps onto the reduced intersection algebra, and its kernel is \(uv^2\mathbb{C}[v]/(v^{J-2})\). With \(t=v,n=u\), this is exactly \(nt^2T\), as asserted in \eqref{cls:scalar-orders}.
\end{proof}

\subsubsection{Case (viii)}

The last row follows from the same central coefficient calculation, without connecting components in the endpoint piece.

\begin{lemma}\label{app:local:case-eight}
The section and function restriction images in case (viii) satisfy \eqref{rep:nil-table} and \eqref{cls:scalar-orders} for every arm length and self-intersection allowed in that case.
\end{lemma}

\begin{proof}
Write \(W=4P+2B_1+\sum_{i=1}^l\min(4,l+1-i)A_i\), where \(l\geq3\), \(A_1^2=-2\), and \(B_1^2=-p\). Use \(\mathbb{C}[u,v]/(u^2,v^4)\) at the external attachment. Lemma~\ref{star:nonuniform-arm} imposes a zero at zero in central degrees two and three and, if \(p=2\), the additional condition \(g_3'(0)=2f_3'(0)\). At the long arm, degree one requires a common zero at one; degree two requires a double zero there, because tangent exponent one gives the sequence \((2,1,0,-1)\). In these two degrees, every tangent exponent at least the indicated vanishing order is permitted, either by nonnegative propagation or by reaching a cycle multiplicity.

Degree zero gives \(\mathbb{C}e_n\). Degree one is the \((b,s)=(2,1)\) calculation above: its reduced coefficient map is invertible, so its nilpotent kernel is zero. In degree two the exact solution is
\[
f_2=z(z-1)^2(az+b),\qquad g_2=cz(z-1)^2.
\]
Its outgoing image is
\[
v^2\left(a+2c+u(b-2a-6c),-c+3cu\right).
\]
Reduction modulo \(u\) is onto \(v^2\mathbb{C}^2\), and its kernel is \(uv^2\mathbb{C}e_n\). In degree three the family
\[
f_3=z(z-1)^3(az^2+bz+e/2),\qquad g_3=z(z-1)^3(dz+e)
\]
satisfies the short-arm condition and vanishes to order three at the long arm. Its exponents there are nonnegative and nondecreasing, so it extends for every permitted arm. Its outgoing jet coefficients are
\[
(a+2d,\ b-3a+2e-8d,\ -d,\ 4d-e).
\]
This map is invertible: the last two coordinates determine \(d,e\), and the first two then determine \(a,b\). Thus degree three supplies all \(v^3B^{\oplus2}\), where \(B=\mathbb{C}[u]/(u^2)\). With \(t=v,n=u\), these images give precisely the reduced image and nilpotent kernel in the last row of \eqref{rep:nil-table}.

For functions, central degree one is \(\mathbb{C}z(z-1)\). In degrees two and three the extendable families \(z(z-1)^2\mathbb{C}[z]_{\leq1}\) and \(z^2(z-1)^3\mathbb{C}[z]_{\leq1}\) give the entire outgoing targets. Hence the function image is \(\mathbb{C}+\mathbb{C}v(1-u)+v^2B+v^3B\), which is onto the reduced algebra and has kernel \(uv^2B+uv^3B\).
\end{proof}

\section{Representations and interpolation}\label{sec:interpolation}\label{sec:quiver-interpolation}

We express \(H^1(Z,S_X|_Z)\) as an extension group and compute it by interval decomposition and interpolation on the connecting chain.

\subsection{Restriction representations}
\label{fnd:representations}\label{fnd:fixed-quiver}

Let \(\Lambda\) be the path algebra of the quiver
\[
1 \to 2 \leftarrow 3 \to 4 \leftarrow 5.
\]
For \(1 \leq i \leq j \leq 5\), let \(I[i,j]\) have a one-dimensional space at vertices \(i,\ldots,j\), zero elsewhere, and identity maps between its nonzero spaces. Put \(J_0 = I[1,5]\). The algebra \(\Lambda\) and the module \(J_0\) have dimensions nine and five.

For the cycle \(Z = W_L+2K+W_R\) constructed above, put \(D_i = W_i\cap2K\) and \(P_i = H^0(D_i,S_X|_{D_i})\). Define
\[
V_i = \operatorname{im}\left(H^0(W_i,S_X|_{W_i}) \to P_i\right),
\quad
V_K = \operatorname{im}\left(H^0(2K,S_X|_{2K}) \to P_L\oplus P_R\right).
\]
The inclusions and coordinate projections define a representation
\[
M_0 = \left(V_L \to P_L \leftarrow V_K \to P_R \leftarrow V_R\right).
\]
In particular, \(V_K\) is the simultaneous image of the two restrictions. If \(K\) is empty, use the representation of the two pieces and their single intersection at vertices one, two, and three, and put zero at vertices four and five.

\begin{proposition}\label{fnd:representation}\label{rep:ext-comparison}
For every analytic realization and every cycle \(Z\) constructed above, there is a natural isomorphism
\begin{equation}\label{fnd:ext-comparison}
H^1(Z,S_X|_Z) \simeq \operatorname{Ext}^1_\Lambda(J_0,M_0).
\end{equation}
The cycle \(Z\) need not be anti-nef.
\end{proposition}

\begin{proof}
The module \(J_0\) has the projective resolution
\begin{equation}\label{fnd:projective-resolution}
0 \to \Lambda e_2\oplus\Lambda e_4
 \to \Lambda e_1\oplus\Lambda e_3\oplus\Lambda e_5
 \to J_0 \to 0.
\end{equation}
At each source the last map is the identity. At a sink its map is addition from the two incident copies of \(\mathbb{C}\), and the first map sends a scalar to its two opposite values. This verifies exactness at every vertex. Applying \(\operatorname{Hom}_\Lambda(-,M_0)\) identifies the extension group with the cokernel of
\[
V_L\oplus V_K\oplus V_R \to P_L\oplus P_R,
\quad
(v_L,v_K,v_R) \mapsto 
\left(v_L-\operatorname{pr}_L(v_K),v_R-\operatorname{pr}_R(v_K)\right).
\]
The union sequence
\[
0 \to S_X|_Z \to 
S_X|_{W_L}\oplus S_X|_{2K}\oplus S_X|_{W_R}
 \to S_X|_{D_L}\oplus S_X|_{D_R} \to 0
\]
is exact: locally at each intersection it is the sequence for two Cartier equations, and there is no triple intersection. Lemma~\ref{rep:piece-vanishing} gives vanishing of first cohomology on all three pieces. The union sequence therefore identifies \(H^1(Z,S_X|_Z)\) with the same cokernel. Replacing global section spaces by the images of the restriction maps does not change this cokernel. If the two endpoint pieces meet directly, the corresponding union sequence and the representation supported on vertices one, two, and three give the same proof.
\end{proof}

By Ringel~~\cite[p.~1, Theorem]{Ringel2013}, every finite-dimensional representation of a quiver of type \(A_n\) is a direct sum of interval representations. Its multiplicities are unique by the Krull--Schmidt theorem.

\begin{proposition}\label{thm:three-intervals}\label{prop:multiplicities}
Let
\[
M = \left(M_1\xrightarrow{a}M_2\xleftarrow{b}M_3\xrightarrow{c}M_4\xleftarrow{d}M_5\right)
\simeq \bigoplus_{i\leq j} I[i,j]^{\oplus m_{ij}}.
\]
Put \(A = M_2/\operatorname{im}a\), \(B = M_4/\operatorname{im}d\), and let \(u:M_3\to A\), \(v:M_3\to B\) be the induced maps. There is a canonical exact sequence
\[
0 \to \frac{M_3}{\ker u+\ker v}
 \to \operatorname{Ext}^1_\Lambda(J_0,M)
 \to \operatorname{coker}u\oplus\operatorname{coker}v \to 0.
\]
The first map sends the class of \(z\) to the class of \((u(z),0)\) under the identification \(\operatorname{Ext}^1_\Lambda(J_0,M)=(A\oplus B)/\operatorname{im}(u,v)\). The interval multiplicities satisfy
\[
\begin{aligned}
m_{22}&=\dim\operatorname{coker}u,\qquad m_{44}=\dim\operatorname{coker}v,\\
m_{24}&=\dim\frac{M_3}{\ker u+\ker v}\\
&=\operatorname{rank}u+\operatorname{rank}v-\operatorname{rank}(u,v).
\end{aligned}
\]
In particular, \(\dim\operatorname{Ext}^1_\Lambda(J_0,M)=m_{22}+m_{24}+m_{44}\).
\end{proposition}

\begin{proof}
Resolution~\eqref{fnd:projective-resolution} identifies the extension group with \((A\oplus B)/\operatorname{im}(u,v)\). Projection gives the surjection in the statement. Its kernel is the image of \(z\mapsto[(u(z),0)]\): a representative \((u(x),v(y))\) differs from \((u(x-y),0)\) by \((u(y),v(y))\). The kernel of this map is \(\ker u+\ker v\), since \((u(z),0)=(u(w),v(w))\) is equivalent to \(w\in\ker v\) and \(z-w\in\ker u\). This proves exactness.

On an interval summand, \(\operatorname{coker}u\) is one-dimensional exactly for \(I[2,2]\), and \(\operatorname{coker}v\) exactly for \(I[4,4]\). The quotient by \(\ker u+\ker v\) is one-dimensional exactly for \(I[2,4]\). Additivity gives the multiplicity formulas and the dimension of the extension group. The rank identity follows from the dimension formula for a sum of two subspaces.
\end{proof}

When the remaining connecting chain is nonempty, \(a,d,(b,c)\) are injective, so the interval representations \(I[1,1],I[3,3],I[5,5]\), each supported at a single source vertex, do not occur. If \(q_i = \dim H^0(D_i,\mathcal{O}_{D_i})\), then \(P_i\) has dimension \(2q_i\). The sum of the dimensions of the two endpoint restriction images is at most \(2(q_L+q_R)\), and the same bound holds for the image of simultaneous restriction from \(2K\). Since \(q_L+q_R \leq 16\) for the cycles above, this gives \(\dim M_0 \leq 96\). This bound is uniform over all two-node graphs satisfying the local restrictions, independently of the number of components and their self-intersections. When the endpoint cycles intersect and no connecting component remains, the same estimate gives \(\dim M_0 \leq 48\), and only \(I[2,2]\) can contribute to the extension group.

\subsection{\texorpdfstring{Reduction to the ideal of \(K\) in \(2K\)}{Reduction to the ideal of K in 2K}}

Assume the remaining connecting chain \(K\) is nonempty. Write \(p_L,p_R\) for its marked endpoints, and put \(I = \mathcal{O}_K(-K)\), \(S_K = S_X|_K\). We continue to use \(P_i = H^0(D_i,S_X|_{D_i})\). The endpoint restriction images are those of Proposition~\ref{rep:endpoint-images} and Table~\eqref{rep:nil-table}, with endpoint subscripts on their parameters when needed. Lemma~\ref{rep:piece-vanishing} gives \(S_K \simeq \mathcal{O}_K^{\oplus2}\), \(H^1(K,I) = 0\), and the surjectivity of restriction from \(H^0(2K,S_X|_{2K})\) to \(H^0(K,S_K)\), with kernel \(H^0(K,I\otimes S_K)\).

Set \(r_i = 1\) in cases (ii),(iii), and \(r_i = 2\) otherwise. The nilpotent image is either \(N_i = n_it_i^{r_i}T_i^{\oplus2}\), or
\[
N_i = n_it_i^{r_i}\ell_i+n_it_i^{r_i+1}T_i^{\oplus2}
\]
for a line \(\ell_i\). Put \(\epsilon_i = 0\) in the first case and \(\epsilon_i = 1\) in the second. An additional line in the coefficient of \(n_it_i^2\) occurs only in cases (vii) and (viii), by \eqref{rep:nil-table}. In case (vii) the entire maximal initial sequence of connecting components of self-intersection \(-2\) is included in the endpoint piece; in case (viii) that sequence is empty by the arm and chain conditions in Table~\ref{intro:two-node-table}. Thus in either case the adjacent component of \(K\) has self-intersection at most \(-3\).

\begin{proposition}\label{prop:single-end}
For either endpoint, the map
\[
V_i\oplus V_K \to P_i
\]
obtained by taking the difference of the restriction maps is surjective. Consequently
\[
m_{22} = m_{44} = 0,\qquad
\dim\Ext^1_{\Lambda}(J_0,M_0) = m_{24}.
\]
Moreover, if \(u:V_K \to P_L/V_L\) and \(v:V_K \to P_R/V_R\) are the quotient restrictions, then
\[
\operatorname{Ext}^1_{\Lambda}(J_0,M_0) \simeq V_K/(\ker u+\ker v).
\]
\end{proposition}

\begin{proof}
The reduced endpoint image contains every positive-order coefficient, and evaluation of \(H^0(K,S_K)\) supplies the constant coefficient. By Lemma~\ref{rep:piece-vanishing}, these sections lift to \(2K\), so the reduced difference map is surjective. Its nilpotent cokernel is a quotient of \(n_iT_i^{\oplus2}/N_i\), which requires jets of length at most \(r_i+\epsilon_i\).

On an end component \(K_i\) of self-intersection \(-c_i\), the degree of \(I\) is \(c_i-1\), or \(c_i\) if \(K\) has one component. This degree is at least \(r_i+\epsilon_i-1\): the only case requiring degree two has \(r_i=2\), \(\epsilon_i=1\), and \(c_i\geq3\), as observed above. Evaluation on \(K_i\) therefore supplies arbitrary jets of the required length. Here no value is prescribed at the opposite endpoint. Each chosen section extends along the chain by matching one value at each successive attaching point, since all other component degrees of \(I\) are nonnegative. Thus \(H^0(K,I\otimes S_K)\), the kernel of reduction from \(2K\), surjects onto \(n_iT_i^{\oplus2}/N_i\). The difference map is surjective, and Proposition~\ref{prop:multiplicities} gives the remaining assertions.
\end{proof}

\subsection{Transport of the endpoint lines}

We compute determinants of transported endpoint vectors using continuants. These formulas determine when the endpoint lines coincide.

\begin{proposition}\label{int:transport-determinant}
Let \(b_1,\ldots,b_m \geq 2\), \(v = (a,-c)\), and \(v' = (a',-c')\), where \(a,a' > 0\), \(0 \leq c \leq a\), and \(0 \leq c' \leq a'\). Put \(\rho = c/a\), \(\rho' = c'/a'\), and
\begin{equation}\label{int:transport}
\mathcal{T} = T_{b_m}\cdots T_{b_1},\qquad T_b = \begin{pmatrix}b&1\\-1&0\end{pmatrix},\qquad D = \frac{\det(\mathcal{T}v,(-c',a'))}{aa'}.
\end{equation}
Then \(D \geq 0\), with equality if and only if \(\rho = \rho' = 1\) and every \(b_j = 2\). For \(m = 0\), take \(\mathcal{T}\) to be the identity.
\end{proposition}

\begin{proof}
For \(m \geq 1\), let \(A\) be the tridiagonal matrix with diagonal \(b_1,\ldots,b_m\) and adjacent entries minus one, and set \(A_{\rho,\rho'} = A-\rho e_1e_1^{\mathsf{T}}-\rho'e_me_m^{\mathsf{T}}\). The continuant recurrence gives \(D=\det A_{\rho,\rho'}\); when \(m=1\), both sides are \(b_1-\rho-\rho'\). For every \(x \in \mathbb{R}^m\),
\[
x^{\mathsf{T}}A_{\rho,\rho'}x = \sum_{j = 1}^{m-1}(x_j-x_{j+1})^2+\sum_{j = 1}^m(b_j-2)x_j^2+(1-\rho)x_1^2+(1-\rho')x_m^2.
\]
All summands are nonnegative. A nonzero kernel vector must be constant, and exists if and only if both endpoint coefficients \(1-\rho,1-\rho'\) and all \(b_j-2\) vanish. Otherwise the matrix is positive definite. This proves the assertion for \(m \geq 1\), including \(m = 1\), when the two endpoint terms concern the same variable. For \(m = 0\), one has \(D = 1-\rho\rho'\).
\end{proof}

\begin{corollary}\label{int:four-arm-discriminant}
Suppose that both endpoint nodes have self-intersection \(-2\) and that their four outward arms are single components of self-intersections \(-p,-q,-p',-q'\), where \(p,q,p',q' \geq 2\). Put \(\tau = 2-1/p-1/q\), \(\sigma = 2-1/p'-1/q'\), and let \(A_\Gamma\) be the negative of the intersection matrix of the full graph. If \(\mathcal{T} = T_{b_m}\cdots T_{b_1}\) is the transport along the connecting chain, then
\[
\det\left(\mathcal{T}(\tau,-1),(-1,\sigma)\right)
 = \frac{\det A_\Gamma}{pqp'q'}.
\]
If every component of the connecting chain has self-intersection \(-2\), this common value is
\[
(\tau-1)+(\sigma-1)+(m+1)(\tau-1)(\sigma-1).
\]
\end{corollary}

\begin{proof}
The Schur complement of the four leaf entries gives the tridiagonal matrix with diagonal \(\tau,b_1,\ldots,b_m,\sigma\), whose determinant is \(\det A_\Gamma/(pqp'q')\). For \(m\geq1\), eliminating its first and last entries gives \(\tau\sigma\det A_{1/\tau,1/\sigma}\), which is the transported determinant by Proposition~\ref{int:transport-determinant}; for \(m=0\), both are \(\tau\sigma-1\). If all \(b_j=2\), the continuant of \(r\) such entries is \(r+1\), giving the last formula.
\end{proof}

\begin{corollary}\label{int:actual-lines-distinct}
The endpoint lines in \eqref{int:orders} are distinct after transport through the remaining connecting chain by the logarithmic restriction maps, unless both endpoints are case (ii), all four outward leaves have self-intersection \(-2\), and every component of the remaining connecting chain has self-intersection \(-2\). The exceptional graph is not negative definite. The one-dimensional images in the constant coefficient of the reduced restriction in \eqref{rep:nil-table} are always distinct after the same comparison.
\end{corollary}

\begin{proof}
In the convention of Proposition~\ref{int:transport-determinant}, the slopes of the additional lines are
\[
\rho = \frac1{\tau}\quad\text{in case (ii)},\qquad
\rho = 0\quad\text{in cases (iii) and (viii)},\qquad
\rho = \frac23\quad\text{in case (vii)}.
\]
Here \(\tau = 2-1/p-1/q \geq 1\). Thus \(\rho \leq 1\), with equality if and only if the endpoint is case (ii) and \(p = q = 2\). Proposition~\ref{int:transport-determinant} gives the stated exceptional graph. The cycle with multiplicity two on the path between the nodes and multiplicity one on its four leaves has zero intersection with every component, so its intersection matrix is not negative definite. Every line in the constant coefficient of the reduced restriction has slope strictly less than one, including the slopes \(h/(h+1)\) in case (iv), so the same proposition proves the last assertion.
\end{proof}

\begin{proposition}\label{int:reduced}
The quotient of the Mayer--Vietoris complex of the restriction maps by its nilpotent subcomplex, defined as the kernel of reduction modulo the ideal of \(K\), is acyclic in degrees zero and one. Consequently
\begin{equation}\label{int:exact-nil-test}
\operatorname{Ext}^1_{\Lambda}(J_0,M_0)
 \simeq 
\operatorname{coker}\left[
H^0(K,I\otimes S_K) \to 
\frac{n_LT_L^{\oplus2}}{N_L}\oplus
\frac{n_RT_R^{\oplus2}}{N_R}
\right].
\end{equation}
Here \(M_0\) is defined by the simultaneous restriction maps on sections in Proposition~\ref{rep:ext-comparison}.
\end{proposition}

\begin{proof}
Use the two-term complex, placed in degrees zero and one,
\[
\mathcal{C}^{\bullet}
=\left[V_L\oplus V_K\oplus V_R \to P_L\oplus P_R\right],
\]
whose differential is the difference of the restriction maps. Proposition~\ref{rep:ext-comparison} identifies its first cohomology with \(\operatorname{Ext}^1_\Lambda(J_0,M_0)\). Reduction modulo \(n_L,n_R\) gives a quotient complex \(\overline{\mathcal{C}}^{\bullet}\). The reduction of \(V_K\) identifies with \(H^0(K,S_K)\): every section on \(K\) lifts to \(2K\) by Lemma~\ref{rep:piece-vanishing}, and evaluation at either endpoint is injective on \(H^0(K,S_K)\).

The endpoint summands of \(\overline{\mathcal{C}}^{\bullet}\) map isomorphically onto their images in the target. Canceling these summands leaves the complex
\[
H^0(K,S_K) \to
\mathbb{C}^2/\mathbb{C}w_L\oplus
\mathbb{C}^2/\mathbb{C}w_R.
\]
Lemma~\ref{fnd:two-point-transport} identifies the endpoint fibers by the product of the matrices \(T_b\). Under this identification the two lines are distinct by Corollary~\ref{int:actual-lines-distinct}, so the displayed map is an isomorphism. Thus \(\overline{\mathcal{C}}^{\bullet}\) is acyclic.

Put \(\mathcal{N}^{\bullet}=\ker(\mathcal{C}^{\bullet}\to\overline{\mathcal{C}}^{\bullet})\). The exact sequence of complexes identifies \(H^1(\mathcal{C}^{\bullet})\) with \(H^1(\mathcal{N}^{\bullet})\). Its endpoint summands have images \(N_L,N_R\). Its middle summand is the restriction image of \(H^0(K,I\otimes S_K)\): a section on \(2K\) whose endpoint reductions vanish has zero reduction on all of \(K\), by the injectivity just noted. Quotienting the target of this nilpotent complex by \(N_L\oplus N_R\) gives~\eqref{int:exact-nil-test}.
\end{proof}

\subsection{Interpolation at the two ends}

The nilpotent images in \eqref{rep:nil-table} can be written as follows. Assign an integer \(r_i\) and, in some cases, a line \(\ell_i\):
\begin{equation}\label{int:orders}
\begin{array}{c|c|c}
\text{case}&r_i&\text{additional line}\\ \hline
\text{(ii)}&1&\mathbb{C}(\tau_i,-1)\\
\text{(iii), }b_P = 3,\ s_P = 2&1&\mathbb{C}(1,0)\\
\text{(iii), otherwise}&1&\text{none}\\
\text{(i), (iv), (v), (vi)}&2&\text{none}\\
\text{(vii)}&2&\mathbb{C}(3,-2)\\
\text{(viii)}&2&\mathbb{C}(1,0)
\end{array}
\end{equation}
Thus one first prescribes all vector coefficients below order \(r_i\); if a line occurs, one also prescribes the coefficient of order \(r_i\) modulo that line. Write \(K=K_1+\cdots+K_m\), with \(K_j^2=-b_j\). Only components of self-intersection \(-2\) have been absorbed into the endpoint cycles, so \(\sum_j(b_j-2)=d\), with \(d\) as defined for the original connecting chain. Moreover, \(\deg I=-K^2=d+2\). Let \(f\) be the number of additional lines, and put
\[
B = I(-r_Lp_L-r_Rp_R),\qquad
\delta = \deg B = d+2-r_L-r_R.
\]

\begin{lemma}\label{int:jet-evaluation}
Let \(C\) be a reduced rational chain and \(\mathcal{A}\) a line bundle. If \(C\) has at least two components, assume that its internal degrees are nonnegative and its end degrees are at least minus one. Then
\[
h^1(C,\mathcal{A}) = \max\{-\deg\mathcal{A}-1,0\}.
\]
In particular, this applies to \(I(-ap_L-bp_R)\) for any \(a,b \in \{1,2\}\). It also applies to the line bundle summands in \eqref{int:modification}, obtained from \(B\otimes S_K\) by the additional line conditions in \eqref{int:orders}.
\end{lemma}

\begin{proof}
For one component this is the cohomology formula on \(\mathbb{P}^1\). On a chain with more than one component, a section vanishes on an end of degree minus one. The normalization sequence identifies both cohomology groups with those of the remaining chain after subtracting its attaching point. The total degree is unchanged. Repeat while an end component of degree minus one remains. The process either ends with one component, or leaves nonnegative component degrees, in which case normalization gives \(H^1 = 0\). In the latter case the total degree is nonnegative. This proves the formula.

The end degrees of \(I\) are \(b_i-1\), and its internal degrees are \(b_j-2\). Subtracting an order at most two therefore leaves end degrees at least minus one. An additional vanishing condition from a line in \eqref{int:orders} can lower an end degree once more only when the prescribed order is two; the adjacent component of \(K\) then has self-intersection at most \(-3\), as shown before Proposition~\ref{prop:single-end}. Thus every summand used below satisfies the same bound. On a chain consisting of one component there is no restriction on the degree.
\end{proof}

\begin{proposition}\label{int:interpolation}\label{int:chain-excess-reduction}
Assume that the endpoint cycles are those of Proposition~\ref{rep:cycles}, that \(H^1(W_i,S_X|_{W_i}) = 0\), and that the remaining connecting chain is nonempty. For a negative-definite graph,
\begin{equation}\label{int:dimension}
\dim_{\mathbb{C}}\operatorname{Ext}^1_{\Lambda}(J_0,M_0) = 
\begin{cases}
2,&\delta = -2,\\
f,&\delta = -1,\\
0,&\delta \geq 0.
\end{cases}
\end{equation}
In particular, inserting components of self-intersection \(-2\) into the connecting chain leaves the displayed dimension unchanged whenever the resulting graph is negative definite and satisfies the same endpoint hypotheses.
\end{proposition}

\begin{proof}
Let \(\mathcal{F}\) be the kernel of the sheaf morphism inducing the evaluation map in \eqref{int:exact-nil-test}. Since \(H^1(K,I\otimes S_K) = 0\), the cokernel of evaluation on global sections is \(H^1(K,\mathcal{F})\). The kernel of evaluation below the two prescribed orders is \(B\otimes S_K\), where \(\deg B = \delta\). The additional line conditions are elementary modifications at the two smooth marked points of \(K\). Trivialize \(S_K\) globally, so these lines belong to one vector space. A constant basis adapted to them gives
\begin{equation}\label{int:modification}
\begin{array}{c|c}
\text{additional lines}&\mathcal{F}\\ \hline
\text{none}&B\oplus B\\
\text{one, at }p&B\oplus B(-p)\\
\text{two, distinct}&B(-p_L)\oplus B(-p_R)\\
\text{two, equal}&B\oplus B(-p_L-p_R).
\end{array}
\end{equation}
For example, with distinct lines choose one basis vector in each; the coefficient in the opposite line must vanish at each marked point. Corollary~\ref{int:actual-lines-distinct} excludes equal lines in the negative-definite cases. Lemma~\ref{int:jet-evaluation} computes the cohomology of each summand. For \(\delta = -2\), both orders are two and every component of \(K\) has self-intersection \(-2\), so the adjacency requirement excludes any additional line. The dimension is two. For \(\delta = -1\), each additional line contributes one and the dimension is \(f\). For \(\delta \geq 0\), all summands have zero \(H^1\).
\end{proof}

\begin{corollary}\label{int:nonempty-vanishing}
Every graph satisfying the two-node conditions of Theorem~\ref{thm:classification} and having a nonempty remaining connecting chain satisfies
\[
\operatorname{Ext}^1_{\Lambda}(J_0,M_0) = 0
\]
for every analytic realization.
\end{corollary}

\begin{proof}
The extension group can be nonzero only when \(\delta = -2,-1\). If \(r_L+r_R = 4\), then \(\delta = d-2\). The value \(d = 0\) is excluded by the restriction in \eqref{rep:condition-b} at a node of self-intersection \(-2\) with one short and one long outward arm. If \(d = 1\), an additional line occurs only in case (vii) or (viii), and the condition in the last column of Table~\ref{intro:two-node-table} excludes both because the opposite node has self-intersection \(-2\) and one long outward arm. Thus \(f = 0\).

If \(r_L+r_R=3\), only \(d=0\) needs consideration. The order-two endpoint is then case (i): cases (iv), (vi), and (vii) include the maximal initial sequence of components of self-intersection \(-2\), whereas cases (v) and (viii) require the next component to have self-intersection at most \(-3\). The restriction in \eqref{rep:condition-b} makes the opposite endpoint case (iii) with a surjective order-one coefficient map. Thus \(f=0\). If \(r_L+r_R=2\), then \(\delta=d\geq0\). Proposition~\ref{int:interpolation} proves every case.
\end{proof}

\subsection{Direct intersections of the two endpoint pieces}

When no connecting component remains, the two pieces meet in one intersection scheme. Except when both ends are case (ii), the restriction in \eqref{rep:condition-b} leaves a case (iii) endpoint. Write its node self-intersection as \(-b\), where \(b \geq 3\), and let \(s\) be its number of long outward arms. The intersection ring is \(\mathbb{C}[t,n]/(t^J,n^2)\), where \(n\) is normal to this node and \(t\) is tangent to it. The ordered logarithmic basis is \((t\partial_t,n\partial_n)\). Its reduced restriction image is \(\mathbb{C}(0,1)\), whereas the reduced restriction image of the other endpoint is \(\mathbb{C}w+tT^{\oplus2}\) with first coordinate of \(w\) nonzero. Their reduced differential is an isomorphism. The cokernel of the restriction complex is consequently the cokernel of the subcomplex consisting of kernels of reduction, by the same long exact cohomology sequence as in Proposition~\ref{int:reduced}.

\begin{lemma}\label{int:U-jets}
In case (iii), without a line bundle twist, evaluation onto vector-valued jets of length two is surjective if and only if \(b \geq 4\), or \(b = 3\) and \(s \leq 1\). For \(b=3,s=2\), after setting the order-zero coefficient to zero, the order-one image is the normal Euler line.

Evaluation onto jets of length three is surjective for \(b\geq5\), and for \(b=4,s\leq1\). For \(b=4,s=2\), its cokernel is one-dimensional; after the coefficients of orders zero and one have been set to zero, the order-two image is the normal Euler line. For \(b = 3\), the image modulo \(nt^3\) of sections vanishing on the reduced node has dimension \(5-s\).

When \(b = 3,s = 0\), set the coefficients of orders zero and one to zero. The annihilator of the resulting order-two image is spanned by the linear functional on \((\alpha,\beta)\) given by
\begin{equation}\label{int:U-functional}
\beta+\left(3-\frac1p-\frac1q\right)\alpha,
\end{equation}
where \(-p,-q\) are the self-intersections of its two outward arms, each consisting of one component.
\end{lemma}

\begin{proof}
By Proposition~\ref{app:local:doubled-node}, the conditions at the two outward arms in order one in the normal coordinate give the bundle
\[
\mathcal{O}_{\mathbb{P}^1}(b-2)\oplus\mathcal{O}_{\mathbb{P}^1}(b-s-1).
\]
Evaluation on its global sections at the marked endpoint onto jets of length \(a\) has rank \(\min(b-1,a)+\min(b-s,a)\). For \(a = 2\), this is four except when \(b = 3,s = 2\), when it is three; the corresponding row of \eqref{rep:nil-table}, with the two branches interchanged, identifies the remaining order-one image with the normal Euler line. For \(a=3\), the rank is six for \(b\geq5\) and for \(b=4,s\leq1\), five for \(b=4,s=2\), and \(5-s\) for \(b=3\).

For \(b=4,s=2\), both polynomials in \eqref{app:local:node-formula} are divisible by \(x(x-1)\). Vanishing of the first two outgoing coefficients gives \(\deg f\leq2\) and \(\deg g\leq1\), by \eqref{app:local:infinity}. Thus \(g=0\), whereas \(f\) may be a multiple of \(x(x-1)\). The resulting order-two image is the normal Euler line, as asserted.

For the last assertion, put the attachment points of the two short arms at \(x = 0,1\). Write the section of order one in the normal coordinate using
\[
f = ax^3+cx^2+ex+h,\qquad
g = dx^2+zx+j.
\]
The leaf equations are \(j = ph\) and \(d+z+j = -q(a+c+e+h)\). Changing to the infinity chart, the first three coefficient columns are
\[
(-d,a+3d),\qquad(d-z,c+3(z-d)),\qquad(z-j,e+3(j-z)).
\]
Vanishing of the first two gives \(a = c = d = z = 0\). The last is \((\alpha,\beta) = (-j,e+3j)\), while the leaf equations give \(e = -j(1/p+1/q)\). Thus this image is the kernel of the functional \eqref{int:U-functional}.
\end{proof}

For an opposite endpoint at \(P\), the nilpotent quotient has the following form, by Remark~\ref{rep:direct-endpoint}. A line in the last column is quotiented from the highest coefficient of the indicated vector-valued jets; all lower coefficients are unrestricted.
\begin{equation}\label{int:direct-targets}
\begin{array}{c|c|c}
\text{opposite endpoint}&\text{jet length}&\text{line in the highest coefficient}\\ \hline
\text{(iii), }(b_P,s_P)\neq(3,2)&1&\text{none}\\
\text{(i), (iv), (v), (vi)}&2&\text{none}\\
\text{(ii)}&2&\mathbb{C}(\tau,-1)\\
\text{(iii), }b_P=3,\ s_P=2&2&\mathbb{C}(1,0)\\
\text{(vii)}&3&\mathbb{C}(3,-2)\\
\text{(viii)}&3&\mathbb{C}(1,0)
\end{array}
\end{equation}
When the lower coefficients vanish, a change preserving the ordered branches multiplies the leading coefficient vector by a common nonzero scalar, by Proposition~\ref{rep:endpoint-images}. Thus the lines in the table can be compared in the ordered basis fixed above.

\begin{proposition}\label{int:direct}
For every analytic realization satisfying the two-node conditions of Theorem~\ref{thm:classification} and having no remaining connecting component, \(\operatorname{Ext}^1_{\Lambda}(J_0,M_0) = 0\).
\end{proposition}

\begin{proof}
Except when both endpoints are case (ii), choose a case-(iii) endpoint as above and use Table~\eqref{int:direct-targets}. Length-one evaluation is always surjective. For the second row, Lemma~\ref{int:U-jets} gives surjectivity by \eqref{rep:condition-b}. For the next two rows, the only possible cokernel of length-two evaluation occurs at \(b=3,s=2\); it is the tangent coefficient in order one. Both displayed lines have nonzero tangent coefficient and therefore remove this cokernel.

For the last two rows, length-three evaluation is surjective when \(b\geq5\). At \(b=4\), Lemma~\ref{int:U-jets} shows that the only possible cokernel is the tangent coefficient in order two, and both displayed lines remove it. At \(b=3\), the graph conditions require \(s=0\). Formula~\eqref{int:U-functional} is nonzero on \((1,0)\), and on \((3,-2)\) its value is
\[
7-\frac3p-\frac3q \geq 4.
\]
The additional line therefore surjects onto the one-dimensional cokernel.

It remains to consider two case-(ii) endpoints. Write their parameters as \(\tau,\sigma \geq 1\). After eliminating the reduced differential and the constant coefficient of the nilpotent restriction, the remaining quotient is \(\mathbb{C}^2\) modulo the sum of the lines spanned by \((\tau,-1)\) and \((-1,\sigma)\). Their determinant is \(\tau\sigma-1\). It vanishes only if all four outward leaves have self-intersection \(-2\), when the graph is not negative definite. Thus the difference of the two restriction maps is surjective in every admissible case.
\end{proof}

\subsection{The last necessary inequality}

We prove the last two-node restriction on a cycle contained in \(5E\). The same calculation gives nonvanishing of the extension group and supplies the final obstruction used in Theorem~\ref{thick:criterion}.

\begin{proposition}\label{int:critical-necessity}
Assume that every component has self-intersection at most \(-2\) and that the restrictions stated in Proposition~\ref{rep:local-restrictions} hold. If an endpoint has case \textup{(vii)} or \textup{(viii)}, \(s_Q\geq1\), and \(b_Q+d<4\), then there is an effective cycle \(Y\leq Z\), with \(Y\leq5E\), such that \(\chi(Y,S_X|_Y)=-1\). Consequently, \(H^1(Y,S_X|_Y)\neq0\) and \(\operatorname{Ext}^1_{\Lambda}(J_0,M_0)\neq0\) for every analytic realization. Anti-nefness of \(Z\) is not required.
\end{proposition}

\begin{proof}
The local restrictions give \(s_Q\leq1\) when \(b_Q=2\), while \eqref{rep:condition-b} excludes \((b_Q,d)=(2,0)\) and \((b_Q,d,s_Q)=(3,0,2)\). Hence
\begin{equation}\label{thick:failure-triples}
(b_Q,d,s_Q)=(3,0,1)\quad\text{or}\quad(2,1,1).
\end{equation}
Write the outward arms at \(P\) as \(B\) and \(A_1+\cdots+A_\ell\), and let \(D_1,D_2\) be the first components of the outward arms at \(Q\). Denote the entire connecting chain by \(C_1,\ldots,C_h\), read from \(P\), with \(C_i^2=-c_i\) and \(d=\sum_i(c_i-2)\). Define \(Y\) by
\[
\begin{aligned}
\text{(viii)}:\quad Y&=3P+B+2A_1+A_2+2\sum_{i=1}^h C_i+2Q+D_1+D_2,\\
\text{(vii)}:\quad Y&=5P+2B+3A_1+A_2+4C_1+3C_2+2\sum_{i=3}^h C_i+2Q+D_1+D_2.
\end{aligned}
\]
An empty sum is zero. In case \textup{(viii)}, \(\ell\geq3\), so \(A_2\) has valency two. In case \textup{(vii)}, \(h\geq2\) and \(c_1=c_2=2\). Table~\ref{rep:case-table} gives \(Y\leq Z\) in both cases; the displayed multiplicities also give \(Y\leq5E\).

For a rank-two bundle \(\mathcal{G}\), Riemann--Roch gives \(\chi(Y,\mathcal{G}|_Y)=-Y\cdot(Y+K_X)+Y\cdot\det\mathcal{G}\); see Wahl~~\cite[proof of Proposition~2.6]{Wahl2016}. Since \(\det S_X=\mathcal{O}_X(-K_X-E)\), adjunction gives, for \(Y=\sum_vy_vE_v\),
\begin{equation}\label{thick:euler-formula}
\chi(Y,S_X|_Y)=\sum_v\left[b_vy_v(y_v-1)+(4-t_v)y_v\right]-2\sum_{\{v,w\}\text{ an edge}}y_vy_w,
\end{equation}
where \(E_v^2=-b_v\) and \(t_v\) is its valency in the graph of \(E\). Group the terms in this formula into the two endpoint configurations and the connecting terms. Include internal edges with their endpoint configuration, and include the two attaching edges with the connecting chain. Direct substitution gives
\[
\begin{array}{c|c|c|c}
\text{case}&\text{near }P&\text{connecting terms}&Q+D_1+D_2\\ \hline
\text{(viii)}&6&2d-12&2b_Q-s_Q\\
\text{(vii)}&20&2d-26&2b_Q-s_Q
\end{array}
\]
If the connecting chain is empty in case \textup{(viii)}, its column consists of the edge joining \(P\) and \(Q\), which contributes \(-12\). Both rows give
\begin{equation}\label{thick:critical-euler}
\chi(Y,S_X|_Y)=-6+2b_Q-s_Q+2d=-1,
\end{equation}
by \eqref{thick:failure-triples}. Thus \(h^1(Y,S_X|_Y)=h^0(Y,S_X|_Y)+1>0\). Since \(Y\leq Z\), Lemma~\ref{fnd:restriction} gives nonvanishing on \(Z\). Lemma~\ref{rep:piece-vanishing} and Proposition~\ref{rep:ext-comparison} identify this cohomology with the stated extension group without using anti-nefness.
\end{proof}

\section{Vanishing and the classification theorem}\label{sec:classification-proof}\label{sec:classification}

We extend the vanishing on \(Z\) to the resolution, complete the classification, and prove the optimal criterion on \(5E\).

\subsection{\texorpdfstring{Restriction maps for \(\mathcal{O}_Z\)}{Restriction maps for the structure sheaf}}

Use the decomposition of \(Z\) from Proposition~\ref{rep:cycles}, the endpoint algebras \(A_i\subset Q_i=T_i[n_i]/(n_i^2)\), and the scalar orders \(\rho_i\) of Table~\eqref{rep:nil-table}. If the remaining connecting chain \(K\) is nonempty, put \(I=\mathcal{O}_K(-K)\).

\begin{proposition}\label{cls:scalar}
For every two-node graph satisfying the two-node conditions of Theorem~\ref{thm:classification},
\[
h^1(Z,\mathcal{O}_Z) = 
\begin{cases}
1,&\text{case (ii) at both ends, all connecting self-intersections }-2,\\
0,&\text{otherwise}.
\end{cases}
\]
The first alternative includes the case in which no connecting component lies between the nodes.

If \(K\neq0\), the more precise formula is
\[
h^1(Z,\mathcal{O}_Z)
=\max\{\rho_L+\rho_R-d-3,0\},
\qquad d=\sum_j(b_j-2),
\]
where the remaining connecting components have self-intersections \(-b_j\).
\end{proposition}

\begin{proof}
Suppose first that \(K\neq0\). Lemma~\ref{rep:piece-vanishing} gives \(H^1(W,\mathcal{O}_W) = H^1(2K,\mathcal{O}_{2K}) = 0\) for \(W = W_L,W_R\). The Mayer--Vietoris sequence for \(\mathcal{O}_Z\) therefore computes \(H^1(Z,\mathcal{O}_Z)\) as the cokernel of the restriction maps on functions. Since \(A_i \to T_i\) is onto, inclusion of the square-zero ideal \(n_iT_i\) induces
\[
Q_i/A_i \simeq n_iT_i/(A_i\cap n_iT_i).
\]
This quotient isomorphism does not require a splitting of the algebra extension. Functions on \(2K\) are constants plus sections of \(I\); constants lie in both endpoint algebras. Thus
\begin{equation}\label{cls:scalar-evaluation}
H^1(Z,\mathcal{O}_Z) \simeq 
\operatorname{coker}\left[
H^0(K,I) \to I|_{\rho_Lp_L}\oplus I|_{\rho_Rp_R}
\right].
\end{equation}
The cokernel in \eqref{cls:scalar-evaluation} is \(H^1(K,I(-\rho_Lp_L-\rho_Rp_R))\). Lemma~\ref{int:jet-evaluation} gives its dimension as \(\max\{\rho_L+\rho_R-d-3,0\}\), where \(d = \sum_j(b_j-2)\). Since \(\rho_L,\rho_R \leq 2\) and \(d \geq 0\), this dimension is one if both vanishing orders are two and every connecting curve has self-intersection \(-2\), and is zero otherwise.

If every component of \(K\) has self-intersection \(-2\) and neither end has case (iii), the graph conditions force both ends to have case (ii). A node of self-intersection \(-2\) with one outward arm of length at least two would require the opposite node to have self-intersection at most \(-3\). In cases (iv), (vi), and (vii), the initial connecting components of self-intersection \(-2\) are included in \(W_L\) or \(W_R\). Thus these cases leave no nonempty \(K\) whose components all have self-intersection \(-2\).

Suppose now that \(K=0\). By the observation preceding Proposition~\ref{rep:cycles}, either one endpoint has case \textup{(iii)} or both have case \textup{(ii)}. In the first alternative, choose \(n\) normal to the node of case \textup{(iii)}. The restriction image of functions vanishing on the reduced node maps surjectively to \(nT/(nt^2T)\): the corresponding line bundle on the node has degree \(b-2 \geq 1\), and its sections extend along the outward chains, where the line bundle has nonnegative component degrees. The function algebra from the opposite end maps onto \(T\) and contains \(nt^2T\) in its intersection with \(nT\). The sum of the two images is therefore the entire intersection algebra.

If both endpoints have case \textup{(ii)}, the intersection ring is \(\mathbb{C}[t,n]/(t^2,n^2)\). The two function images have the form \(\mathbb{C}+\mathbb{C}(t+\alpha tn)\) and \(\mathbb{C}+\mathbb{C}(n+\beta tn)\), after rescaling their generators. Their intersection consists of the constants, so their sum has dimension three and its cokernel has dimension one.
\end{proof}

\begin{proposition}\label{cls:nonexceptional-vanishing}
Let \((X,E)\) be any analytic realization of a two-node graph satisfying the conditions of Theorem~\ref{thm:classification}, other than the exceptional family in Proposition~\ref{cls:scalar}. For the anti-nef cycle \(Z\) of Proposition~\ref{rep:cycles},
\[
H^1(X,S_X) = 0
\quad \Longleftrightarrow \quad
H^1(Z,S_X|_Z) = 0
\quad \Longleftrightarrow \quad
\operatorname{Ext}^1_{\Lambda}(J_0,M_0) = 0.
\]
\end{proposition}

\begin{proof}
Proposition~\ref{rep:ext-comparison} gives the comparison with the extension group. By Proposition~\ref{cls:scalar}, \(H^1(Z,\mathcal{O}_Z) = 0\). Anti-nefness makes \(\mathcal{O}_Z(-kZ)\) have nonnegative component degrees. By Lemma~\ref{fnd:nef-twists}, if \(H^1(Z,S_X|_Z) = 0\), then \(H^1(Z,S_X(-kZ)|_Z) = 0\) for all \(k \geq 0\). So Lemma~\ref{fnd:formal-vanishing} gives \(H^1(X,S_X) = 0\). Conversely, restriction is onto on first cohomology by Lemma~\ref{fnd:restriction}.
\end{proof}

\subsection{The nonrational family with four arms of length one}

Let both nodes have self-intersection \(-2\), let their four outward arms be single curves \(B_1,\ldots,B_4\) with \(B_i^2 = -p_i\), and let every connecting curve have self-intersection \(-2\). Denote by \(T\) the reduced path between the nodes, including both nodes. Put
\[
E = T+\sum_{i = 1}^4B_i,\qquad Z = E+T.
\]
Negative definiteness excludes \(p_1 = \cdots = p_4 = 2\): the cycle with multiplicity two on \(T\) and one on each leaf would otherwise have zero intersection with every component. Thus at least one \(p_i > 2\). Directly,
\begin{equation}\label{cls:four-short-intersections}
Z\cdot T_j = 0,\qquad Z\cdot B_i = 2-p_i \leq 0.
\end{equation}
In particular \(Z\) is anti-nef.

\begin{lemma}\label{cls:four-short-boundary}
Put \(F=S_X(-T)|_E\). At either end of \(T\), let \(R=P+B+B'\) consist of the node and its two outward leaves, with \(B^2=-p\) and \((B')^2=-q\). Then \(H^1(R,F|_R)=0\), and the image of restriction from \(H^0(R,F|_R)\) to the fiber at the attachment to the rest of \(T\) is
\[
\mathbb{C}(1,-\tau),\qquad
\tau=2-\frac1p-\frac1q,
\]
in the ordered tangent and normal logarithmic basis on \(P\), after choosing a frame for \(\mathcal{O}_P(-T)\). These assertions hold for every analytic realization.
\end{lemma}

\begin{proof}
The component restrictions of \(F\) on \(P,B,B'\) are respectively \(\mathcal{O}(1)\oplus\mathcal{O}\), \(\mathcal{O}(-1)\oplus\mathcal{O}\), and \(\mathcal{O}(-1)\oplus\mathcal{O}\), so their first cohomology vanishes. Place the two outward attachment points on \(P\) at \(z=0,1\), and the path attachment at \(u=1/z=0\). The first-order normalization of Lemma~\ref{fnd:two-point-transport} applies while preserving these three marked branches: its obstruction lies in
\[
H^1\left(\mathbb{P}^1,\Theta_{\mathbb{P}^1}(-[0]-[1]-[\infty])\otimes\mathcal{O}_{\mathbb{P}^1}(2)\right)
=H^1(\mathbb{P}^1,\mathcal{O}_{\mathbb{P}^1}(1))=0.
\]
Thus the normal transition is \(y'=z^2y\). Let \(e\) be the frame of \(\mathcal{O}_P(-T)\) induced by \(y\), and let \(e_\infty\) be induced by \(uy'\) at the path attachment; then \(e=ue_\infty\). The sections of \(F|_P\) are
\[
e\left(\alpha z(z-1)\partial_z+(\beta z+\gamma)y\partial_y\right),
\qquad \alpha,\beta,\gamma \in \mathbb{C}.
\]
The tangent quotient is \(\mathcal{O}_P\), and the normal Euler subbundle is \(\mathcal{O}_P(1)\); their sections give the constant coefficient \(\alpha\) and the linear expression \(\beta z+\gamma\). At the path attachment, the value in the stated basis and frame \(e_\infty\) is
\[
(-\alpha,\beta+2\alpha).
\]

On the leaf of self-intersection \(-p\), choose tangent coordinate \(w\), with the attachment at \(w=0\), and normal coordinate \(r\). The transition \(r'=w^pr\) shows that the sections of \(S_X\) restricted to the leaf and vanishing at the attachment are \(w\lambda(w\partial_w-pr\partial_r)\), where \(\lambda\in\mathbb{C}\). Dividing by the frame \(w\) of \(\mathcal{O}_{B}(-T)\) and exchanging tangent and normal entries gives the line \(\mathbb{C}(1,-1/p)\) in the node fiber. The two matching equations are therefore
\[
\gamma=\alpha/p,\qquad \beta+\gamma=-\alpha/q.
\]
They are independent. Thus \(\beta=-\alpha(1/p+1/q)\), and the remaining path value is the line
\[
\mathbb{C}(1,-\tau),\qquad
\tau = 2-\frac1p-\frac1q.
\]
The normalization map for \(R\) has source dimension five and target dimension four. The two independent leaf equations leave the one-dimensional kernel just computed, so the map is surjective and \(H^1(R,F|_R)=0\). Its kernel has the asserted boundary image. The calculation uses only first-order normal transitions preserving the marked branches, and therefore holds for every analytic realization.
\end{proof}

\begin{lemma}\label{cls:four-short-reduced}
For this family,
\[
H^1(T,S_X|_T)=0,\qquad
H^1(E,S_X(-T)|_E)=0.
\]
\end{lemma}

\begin{proof}
Let \(m\) be the number of internal components of \(T\). The restrictions of \(S_X\) to the nodes are \(\mathcal{O}\oplus\mathcal{O}(-1)\), and its restrictions to the internal components are trivial of rank two. The normalization map therefore has source and target of equal dimension. At the left node the Euler line, after exchanging tangent and normal entries at the first crossing, is represented by \((1,0)\). Its transport to the right node is
\[
T_2^m(1,0)=(m+1,-m).
\]
The right node's Euler line is represented by \((0,1)\); the determinant is \(m+1\neq0\), also when \(m=0\). The normalization map is thus an isomorphism, proving \(H^1(T,S_X|_T)=0\).

For the second assertion, put \(F=S_X(-T)|_E\). Lemma~\ref{cls:four-short-boundary} eliminates the two outward leaves at each node and supplies its boundary line. On each internal component \(F\) is trivial of rank two, and \(\mathcal{O}(-T)\) has multidegree zero along the internal chain. A trivialization of this line bundle makes the transport there the same as for \(S_X\). With the fiber conventions of Corollary~\ref{int:four-arm-discriminant}, the two boundary lines admit representatives \((\tau,-1)\) and \((-1,\sigma)\). Their transported determinant is
\[
\frac{\det A_\Gamma}{p_1p_2p_3p_4} > 0,
\]
where \(A_\Gamma\) is the negative intersection matrix. The normalization map after eliminating the leaves has source and target of equal dimension, so this nonzero determinant makes it an isomorphism. This proves \(H^1(E,F)=0\).
\end{proof}

\begin{proposition}\label{cls:four-short}
Every graph in this family satisfies \(H^1(X,S_X) = 0\) for every analytic realization.
\end{proposition}

\begin{proof}
For every \(k \geq 0\), the cycle sequence is
\[
0 \to S_X(-T-kZ)|_E
 \to S_X(-kZ)|_Z
 \to S_X(-kZ)|_T
 \to 0.
\]
By \eqref{cls:four-short-intersections}, the line bundle \(\mathcal{O}_T(-kZ)\) has multidegree zero. It is trivial because \(T\) is a rational tree. Thus the right term has zero \(H^1\) by Lemma~\ref{cls:four-short-reduced}. The left term is \(S_X(-T)|_E\) tensored with the nef line bundle \(\mathcal{O}_E(-kZ)\) on the reduced rational tree \(E\). Lemma~\ref{fnd:nef-twists}, applied to \(E\), gives its \(H^1\)-vanishing for every \(k\). Hence \(H^1(Z,S_X(-kZ)|_Z) = 0\) for all \(k\). Lemma~\ref{fnd:formal-vanishing} finishes the proof.
\end{proof}

\subsection{Completion of the classification}

We combine the necessary restrictions and the vanishing results to prove Theorem~\ref{thm:classification}. The same cycle calculations determine which of the taut graphs are rational.

\begin{proof}[Proof of Theorem~\ref{thm:classification}]
Theorem~\ref{red:reduction} leaves chains, cycles, and trees of rational components with one or two trivalent nodes. Propositions~\ref{red:chains} and \ref{red:cycles} and Theorem~\ref{star:classification} prove the first three classes.

For two nodes, Propositions~\ref{rep:local-restrictions} and \ref{int:critical-necessity} prove necessity. Conversely, Proposition~\ref{rep:cycles} constructs an anti-nef cycle \(Z\). Corollary~\ref{int:nonempty-vanishing} and Proposition~\ref{int:direct} give a vanishing extension group, according as the remaining connecting chain is nonempty or empty. Proposition~\ref{rep:ext-comparison} identifies this group with \(H^1(Z,S_X|_Z)\). Propositions~\ref{cls:nonexceptional-vanishing} and \ref{cls:four-short} then give \(H^1(X,S_X)=0\) for every analytic realization. Theorem~\ref{fnd:tautness} proves tautness.

Every listed graph has an analytic realization by Lemma~\ref{fnd:plumbing}, and its self-intersections ensure that the resolution is minimal and good.
\end{proof}

\begin{corollary}\label{cls:rationality}
Among the graphs in Theorem~\ref{thm:classification}, the nonrational graphs are the cycles and the trees with four arms of length one in which every curve on the path between the nodes, including the nodes, has self-intersection \(-2\). In particular, taut surface singularities need not be rational.
\end{corollary}

\begin{proof}
Chains and the graphs with one node listed in the theorem are rational by Proposition~\ref{red:chains} and Theorem~\ref{star:classification}. Outside the exceptional two-node family, Proposition~\ref{cls:scalar} gives \(H^1(Z,\mathcal{O}_Z) = 0\); anti-nefness, Lemma~\ref{fnd:nef-twists}, and Lemma~\ref{fnd:formal-vanishing}, applied to \(\mathcal{O}_X\), give \(H^1(X,\mathcal{O}_X) = 0\). In the exceptional family, \(h^1(\mathcal{O}_Z) = 1\), so restriction gives \(p_g \geq 1\). For a reduced cycle \(E\), normalization gives \(\chi(\mathcal{O}_E) = \#\{\text{vertices}\}-\#\{\text{edges}\} = 0\) and \(h^0(\mathcal{O}_E) = 1\), hence \(h^1(\mathcal{O}_E) = 1\); restriction again gives \(p_g \geq 1\).
\end{proof}

\subsection{Comparison with Laufer's list}
\label{cls:laufer-comparison}

The families and parameter ranges agree with those in Laufer~~\cite[pp.~136--139]{Laufer1973}; this comparison is not used in the proof. Chains give families \textup{I} and \textup{II}, and cycles give \textup{(V.1)--(V.2)}. For one node, central self-intersection at most \(-3\) gives \textup{(III.1)}; with central self-intersection \(-2\), the three families correspond to \textup{(III.2)--(III.3)}, \textup{(III.4)}, and \textup{(III.5)--(III.9)}, respectively. Terminal parameters equal to two are included. Laufer also requires the components of a good resolution to be smooth.

For two nodes, the endpoint patterns correspond as follows:
\[
\begin{array}{c|cccccccc}
\text{case in Table~\ref{rep:case-table}}&\text{(iii)}&\text{(ii)}&\text{(i)}&\text{(iv)}&\text{(v)}&\text{(vi)}&\text{(vii)}&\text{(viii)}\\ \hline
\text{Laufer's row}&L_1&L_2&L_3&L_4&L_5&L_6&L_7&L_8.
\end{array}
\]
At an end of case \textup{(iii)}, put \(\alpha = 1\) if \((b,s) = (3,2)\), and \(\alpha = 0\) otherwise. Put \(\beta = 1\) if \(b = 3\) and \(s \geq 1\), and \(\beta = 0\) otherwise. The least allowed value of \(d\) is
\[
\begin{array}{c|cccc}
&\text{(ii)}&\text{(i), (iv)--(vi)}&\text{(vii), (viii)}&\text{(iii)}\\ \hline
\text{(ii)}&0&1&1&0\\
\text{(i), (iv)--(vi)}&1&1&2&\alpha\\
\text{(vii), (viii)}&1&2&2&\beta\\
\text{(iii)}&0&\alpha&\beta&0.
\end{array}
\]
In an entry involving case \textup{(iii)}, use its parameters \(b,s\). The local restrictions of Table~\ref{rep:case-table} and negative definiteness are still imposed. The matrix is equivalent to \eqref{rep:condition-b} and the condition in the last column of Table~\ref{intro:two-node-table}, applied in both orientations. In particular, opposite case \textup{(vii)} or \textup{(viii)}, a node with \(b = 3\) requires \(s = 0\) when \(d = 0\); when \(d \geq 1\), either outward arm may have any length.

The joiner \(J_1\) permits an arbitrary chain, whereas \(J_2\) requires an entry at least \(3\). Each of \(L_5,\ldots,L_8\) already contains one such entry. The operation \(C_1\) replaces two such entries by one entry at least \(3\), while \(C_2\) replaces them by one entry at least \(4\). Together with \(C_3\) and \(C_4\), these rules give the displayed bounds on \(d\) from Laufer's Table~\textup{(IV)}. The operations \(C_3,C_4\) permit \(d = 0\) opposite case \textup{(iii)}, with the alternatives encoded by \(\alpha,\beta\). These are the combinatorial substitutions specified by Laufer.

Our case \textup{(iv)} includes the maximal initial sequence of connecting components of self-intersection \(-2\), which Laufer leaves in the joiner. His uncontracted \(L_5,\ldots,L_8\) additionally include the first subsequent component of self-intersection at most \(-3\), whereas we leave that component in the remaining chain. When no such component exists, \(C_3\) or \(C_4\) gives the corresponding full graph. Thus the comparison includes an empty remaining chain and pieces meeting directly. Including components of self-intersection \(-2\) does not change \(d\). Finally, the nonrational exceptions in Laufer~~\cite[p.~139]{Laufer1973} are exactly those in Corollary~\ref{cls:rationality}.

\subsection{\texorpdfstring{A criterion on \(5E\)}{A criterion on 5E}}\label{thick:section}\label{sec:thickenings}

Every excluded graph has an obstruction on a cycle contained in \(5E\). This gives a uniform criterion for tautness.

\begin{theorem}\label{thick:fifth}\label{thick:criterion}
Let \(\Gamma\) be a connected negative-definite graph of a minimal good resolution of a normal complex surface singularity. Then \(\Gamma\) is taut if and only if \(H^1(5E,S_X|_{5E}) = 0\) for every analytic realization \((X,E)\)  of \(\Gamma\).
\end{theorem}

\begin{proof}
If \(\Gamma\) is taut, Theorem~\ref{fnd:tautness} gives \(H^1(X,S_X) = 0\) for every realization. Restriction is onto by Lemma~\ref{fnd:restriction}, proving vanishing on \(5E\).

Conversely, it suffices to find one analytic realization with nonzero restriction cohomology for each excluded graph. The exclusions in the proof of Theorem~\ref{thm:classification} have the following exhaustive bounds on their obstruction cycles.

\begin{center}
\small
\begin{tabular}{@{}>{\raggedright\arraybackslash}p{0.43\linewidth}p{0.12\linewidth}>{\raggedright\arraybackslash}p{0.36\linewidth}@{}}
\toprule
Excluded configuration & Bound & Calculation \\ \midrule
Genus, valency, or reduced incidence & \(Y\leq E\) & Section~\ref{sec:reduction} \\
First restrictions at two nodes & \(Y\leq2E\) & Proposition~\ref{rep:first-layer} and \eqref{rep:condition-b} \\
Arm restrictions at one node & \(Y\leq5E\) & Theorem~\ref{star:classification}, Corollary~\ref{star:central-obstruction} \\
Excluded initial self-intersection sequences & \(Y\leq5E\) & Lemma~\ref{rep:prefix-witness} \\
Last column of Table~\ref{intro:two-node-table} & \(Y\leq5E\) & Proposition~\ref{int:critical-necessity} \\
\bottomrule
\end{tabular}
\end{center}

For the second row, the boundary calculation after Proposition~\ref{rep:first-layer} proves \eqref{rep:condition-b} on \(Z_2\leq2E\). For the third and fourth rows, use the standard plumbing realizations of the one-node graph and of the chosen star-shaped subgraph, respectively. In the latter case, Lemma~\ref{fnd:plumbing} extends the realization, and Lemma~\ref{fnd:forget-branches} preserves the obstruction. Since the chosen reduced subgraph is contained in \(E\), these operations preserve the bound \(Y\leq5E\). The final row gives an obstruction for every realization. Negative-definite chains and cycles that occur as dual graphs of minimal good resolutions are already among the taut classes.

Thus every non-taut graph has an analytic realization and a cycle \(Y \leq 5E\) with \(H^1(Y,S_X|_Y) \neq 0\). Lemma~\ref{fnd:restriction} gives a surjection from \(H^1(5E,S_X|_{5E})\) onto this group, contradicting vanishing for every realization.
\end{proof}

\begin{example}\label{thick:optimality}
Let \(\Gamma_0\) be the star-shaped graph with central weight \(-2\) and arms
\[
[2],\qquad [2,3],\qquad [2,2,2,2,2],
\]
read from the central vertex towards their ends. Then \(\Gamma_0\) is not taut, but
\[
H^1(rE,S_X|_{rE})=0
\qquad\text{for }1 \leq r \leq 4
\]
for every analytic realization \((X,E)\) of \(\Gamma_0\). Its weighted homogeneous realization \((X_{\mathrm{w}},E_{\mathrm{w}})\) satisfies
\[
h^1(5E_{\mathrm{w}},S_{X_{\mathrm{w}}}|_{5E_{\mathrm{w}}})=1.
\]
Consequently, the integer \(5\) in Theorem~\ref{thick:criterion} cannot be replaced by a smaller positive integer.
\end{example}

\begin{proof}
The Seifert invariants are \((n_1,a_1)=(2,1)\), \((n_2,a_2)=(5,3)\), and \((n_3,a_3)=(6,5)\). The graph is negative definite because
\[
2-\frac12-\frac35-\frac56=\frac1{15}>0.
\]
Let \((X_{\mathrm{w}},E_{\mathrm{w}})\) be the minimal good resolution of Pinkham's weighted homogeneous model, and write \(S_{\mathrm{w}}=S_{X_{\mathrm{w}}}\). By Lemma~\ref{star:plumbing-model}, its exceptional neighborhood and grading agree with those of the standard plumbing realization. The integers \(q_k,c_k\) of Section~\ref{star:section} satisfy
\[
q_{k+30}=q_k+2,\qquad c_{k+30}=c_k.
\]
For \(0\leq k<30\), the negative values of \(q_k\) are all \(-1\), at exactly the following indices:
\[
\begin{array}{c|rrrrrrrrr}
k&1&2&3&4&5&7&9&11&17\\ \hline
c_k&1&1&1&0&2&2&1&2&3.
\end{array}
\]
Thus \(q_k\geq-1\) in every nonnegative degree, and \(q_k+c_k<0\) only at \(k=4\), where its value is \(-1\). Proposition~\ref{star:graded} gives \(H^1(X_{\mathrm{w}},\mathcal{O}_{X_{\mathrm{w}}})=0\), and its Euler sequence gives
\[
\dim H^1(X_{\mathrm{w}},L_{X_{\mathrm{w}}})_k
=\dim H^1(X_{\mathrm{w}},S_{\mathrm{w}})_k
=\begin{cases}1,&k=4,\\0,&k\neq4.\end{cases}
\]
Here degree is measured in the normal coordinate of the central component.

Write \(E_{\mathrm{w}}=P+A_1+A_2+A_3\), where \(P\) is central and \(A_i\) are the reduced arms. On \(P\), the successive layers for \(4P\) are
\[
S_{\mathrm{w}}(-jP)|_P\simeq\mathcal{O}_{\mathbb{P}^1}(2j)\oplus\mathcal{O}_{\mathbb{P}^1}(2j-1),\qquad0\leq j\leq3,
\]
so they have vanishing first cohomology. On each arm, \(S_{\mathrm{w}}\) has component restrictions \(\mathcal{O}^{\oplus2}\) or \(\mathcal{O}\oplus\mathcal{O}(1)\), and \(\mathcal{O}_{A_i}(-A_i)\) has nonnegative component degrees. Normalization of the chain therefore gives \(H^1(A_i,S_{\mathrm{w}}(-jA_i)|_{A_i})=0\) for \(j\geq0\). The thickening sequences yield
\[
H^1(4P,S_{\mathrm{w}}|_{4P})=H^1(4A_i,S_{\mathrm{w}}|_{4A_i})=0.
\]
Put \(H=H^1(4E_{\mathrm{w}},S_{\mathrm{w}}|_{4E_{\mathrm{w}}})\) and \(T_i=4P\cap4A_i\). Since the pieces have zero first cohomology, the equivariant Mayer--Vietoris sequence presents \(H\) as a quotient of \(\bigoplus_iH^0(T_i,S_{\mathrm{w}}|_{T_i})\). In coordinates \(P=(x=0)\), \(A_i=(t=0)\), one has
\[
\mathcal{O}_{T_i}\simeq\mathbb{C}\{x,t\}/(x^4,t^4).
\]
The central normal scaling gives \(x\) degree one and \(t\) degree zero; the logarithmic frame \(x\partial_x,t\partial_t\) is invariant. Every monomial in the displayed intersection ring therefore has normal degree \(0,1,2\), or \(3\), so these are the only possible weights of \(H\). On the other hand, restriction gives an equivariant surjection
\[
H^1(X_{\mathrm{w}},S_{\mathrm{w}})\twoheadrightarrow H.
\]
Its source is concentrated in degree \(4\), so \(H\) can have weight only \(4\). The two weight sets are disjoint; hence \(H=0\).

Since \(q_4=-1\) and \(c_4=0\), Corollary~\ref{star:central-obstruction} gives nonzero first cohomology on \(5E_{\mathrm{w}}\). Restriction from the one-dimensional space \(H^1(X_{\mathrm{w}},S_{\mathrm{w}})\) is surjective, so \(h^1(5E_{\mathrm{w}},S_{\mathrm{w}}|_{5E_{\mathrm{w}}})=1\). In particular, \(\Gamma_0\) is not taut.

It remains to prove vanishing on \(4E\) for every analytic realization. Every realization is rational because rationality depends only on the graph. Wahl~~\cite[Corollary~4.5(iii)]{Wahl2015} gives \(h^1(X,L_X)=1\) for all of them. By Wahl~~\cite[Proposition~4.8]{Wahl2015}, a realization which is not quasi-homogeneous satisfies \(h^1(X,S_X)<1\), hence \(h^1(X,S_X)=0\). Restriction then gives vanishing on every exceptional cycle of that realization.

For a quasi-homogeneous realization, Pinkham's data have the form
\[
\mathcal{D}
=B-\frac12p_1-\frac35p_2-\frac56p_3,
\qquad \deg B=2.
\]
An automorphism of \(\mathbb{P}^1\) takes the three marked points to \(0,1,\infty\), and \(B\) is linearly equivalent to \(2[\infty]\). Replacing \(B\) by a linearly equivalent divisor gives an isomorphic graded ring. Thus the quasi-homogeneous realization is unique up to analytic isomorphism, and the preceding calculation applies to it. We have proved vanishing on \(4E\) for every realization. Restriction gives vanishing on \(rE\) for \(1\leq r\leq4\), proving optimality.
\end{proof}

Sch\"uller~~\cite[Corollary~2.11 and Propositions~3.15--3.16]{Schueller2013} gives finite criteria with graph-dependent multiplicities and the intrinsic tangent sheaf of a plumbing scheme. Theorem~\ref{thick:criterion} uses the fixed cycle \(5E\), the ambient logarithmic tangent bundle, and vanishing for every analytic realization.

\end{document}